\documentclass[a4paper,11pt]{amsart}

\calclayout

\newcommand{\sv}{\mathrm{sv}}
\newcommand{\Pet}{\mathrm{Pet}}

\newcommand{\cusp}{\mathrm{cusp}}

\newcommand{\B}{\mathrm{B}}
\newcommand{\M}{\mathrm{M}}

\DeclareRobustCommand\longtwoheadrightarrow
     {\relbar\joinrel\twoheadrightarrow}

\usepackage{amssymb,amsmath,amsthm}
\usepackage{enumerate}
\usepackage{mathtools}

\usepackage[utf8]{inputenc}

\usepackage{color}
\usepackage[colorlinks, linkcolor=blue,citecolor=blue]{hyperref}

\usepackage{tikz-cd}
\tikzcdset{
arrow style=tikz,
diagrams={>={Straight Barb[scale=0.8]}}
}

\def\ord{{\rm ord}}

\def\Sym{\mathop{\rm Sym}}

\def\Hom{\mathop{\rm Hom}}

\def\GL{\mathop{\rm GL}}
\def\im{\mathop{\rm im}}
\def\Res{\mathop{\rm Res}}

\def\Im{\mathop{\rm Im}}
\def\Re{\mathop{\rm Re}}

\def\Spec{\mathop{\rm Spec}}

\def\SL{\mathop{\rm SL}}

\def\id{\text{\rm id}}
\def\an{\text{\rm an}}
\def\dR{\text{\rm dR}}

\def\comp{\text{\rm comp}}

\def\defeq{\coloneqq}

\def\tensor{\otimes}

\def\to{\longrightarrow}
\def\mapsto{\longmapsto}

\newtheorem{theorem}{Theorem}[section]
\newtheorem{prop}[theorem]{Proposition}

\newtheorem{lemma}[theorem]{Lemma}

\newtheorem{coro}[theorem]{Corollary}

\theoremstyle{definition}
\newtheorem{ex}[theorem]{Example}
\newtheorem{defi}[theorem]{Definition}

\newtheorem{obs}[theorem]{Remark}

\numberwithin{equation}{section}

\title{On coefficients of Poincaré series and single-valued periods of modular forms}
\author{Tiago J. Fonseca}

\date{\today}


\address{Mathematical Institute, University of Oxford, Andrew Wiles Building, Radcliffe Observatory Quarter, Woodstock Road, Oxford, OX2 6GG, United Kingdon}
\email{tiago.jardimdafonseca@maths.ox.ac.uk}

\begin{document}
\maketitle

\begin{abstract}
We prove that the field generated by the Fourier coefficients of weakly holomorphic Poincaré series of a given level $\Gamma_0(N)$ and weight $k\ge 2$ coincides with the field generated by the single-valued periods of a certain motive attached to $\Gamma_0(N)$. This clarifies the arithmetic nature of such Fourier coefficients and generalises previous formulas of Brown and Acres--Broadhurst giving explicit series expansions for the single-valued periods of some modular forms. Our proof is based on Bringmann--Ono's construction of harmonic lifts of Poincaré series. 
\end{abstract}

\tableofcontents 

\section{Introduction}

\subsection{Fourier coefficients of Poincaré series}

Let $N\ge 1$ be an integer and consider Hecke's congruence subgroup
$$
\Gamma_0(N) =\left\{\left(\begin{array}{cc}
                            a & b \\
                            c & d
                          \end{array}\right)\in {\SL}_2(\mathbb{Z}) \text{ ; } c \equiv 0 \mod N \right\} \le {\SL}_2(\mathbb{Z})\text{.}
$$
For every integer $m$ and every even integer $k> 2$, we define a \emph{Poincaré series} (cf. \cite{poincare11}, \cite{CS17})
\begin{equation}\label{eq:intro-poincare-def}
P_{m,k,N}(\tau) = \sum_{\gamma\, \in\, \Gamma_\infty\backslash\Gamma_0(N)}\frac{e^{2\pi i m \gamma \cdot \tau}}{j(\gamma,\tau)^k}\text{, }\qquad \tau \in \mathbb{H}\text{,}
\end{equation}
where $\mathbb{H}$ denotes the Poincaré half-plane with its usual $\SL_2(\mathbb{R})$-action $\gamma \cdot \tau = (a\tau+b)(c\tau+d)^{-1}$, $j(\gamma,\tau) = c\tau+d$ is the factor of automorphy, and $\Gamma_{\infty}\le \Gamma_0(N)$ is the stabiliser of the cusp at infinity, given by the condition $c=0$.

The series (\ref{eq:intro-poincare-def}) converges absolutely and uniformly on compact subsets of $\mathbb{H}$ and defines a weakly holomorphic modular form of weight $k$ and level $\Gamma_0(N)$, meaning that
$$
P_{m,k,N}(\gamma\cdot \tau) = j(\gamma,\tau)^kP_{m,k,N}(\tau)
$$
for every $\gamma \in \Gamma_0(N)$, and that $P_{m,k,N}$ is holomorphic on $\mathbb{H}$, and meromorphic at the cusps. When $m=0$, we obtain an Eisenstein series; this case will often be excluded from our discussion. Following Petersson \cite{petersson32}, we shall also consider Poincaré series in weight $k=2$, which are defined via a variation of Hecke's trick.

To a certain extent, the theory of cusp forms reduces to the study of Poincaré series of positive index $m$, as it is well known that $P_{m,k,N}$, for $m>0$, generate the finite-dimensional $\mathbb{C}$-vector space of cusp forms $S_{k}(\Gamma_0(N))$. A natural question is then to understand the Fourier coefficients $a_n(P_{m,k,N})$, defined by
$$
P_{m,k,N} = \sum_{n\ge 1} a_n(P_{m,k,N}) q^n\text{, } \qquad q=e^{2\pi i \tau}\text{.}
$$
When the index $m$ is positive, a classical computation yields the formulas
\begin{equation}\label{eq:intro-Fourier-positive}
  a_n(P_{m,k,N}) = \delta_{m,n}+ 2\pi (-1)^{\frac{k}{2}}\left(\frac{n}{m} \right)^{\frac{k-1}{2}}\sum_{c\ge 1\text{, } N \mid c}\frac{K(m,n;c)}{c}J_{k-1}\left(\frac{4\pi\sqrt{mn}}{c} \right) \text{, }\qquad n\ge 1\text{,}
\end{equation}
where $K(a,b;c) = \sum_{x \in (\mathbb{Z}/c\mathbb{Z})^{\times}}e^{\frac{2\pi i}{c}(ax + bx^{-1})}$ is a Kloosterman sum and $J_{k-1}$ is the $J$-Bessel function of order $k-1$; when the index is negative, we have
$$
P_{-m,k,N} = q^{-m} + \sum_{n\ge 1} a_{n}(P_{-m,k,N})q^n\text{,}\qquad q=e^{2\pi i \tau}\text{,}
$$
with
\begin{equation}\label{eq:intro-Fourier-negative}
  a_n(P_{-m,k,N}) = 2\pi (-1)^{\frac{k}{2}}\left(\frac{n}{m} \right)^{\frac{k-1}{2}}\sum_{c\ge 1\text{, } N \mid c}\frac{K(-m,n;c)}{c}I_{k-1}\left(\frac{4\pi\sqrt{mn}}{c} \right) \text{, }\qquad n\ge 1\text{,}
\end{equation}
where $I_{k-1}$ is an $I$-Bessel function of order $k-1$.

The expressions (\ref{eq:intro-Fourier-positive}) and (\ref{eq:intro-Fourier-negative}) play a prominent role in questions involving growth estimates of Fourier coefficients of modular forms (see \cite{sarnak90}), but are unsatisfying from an algebraic point of view. Although the real numbers $a_n(P_{m,k,N})$, for positive $m$, are expected to be transcendental whenever they are not zero\footnote{Deciding whether a Fourier coefficient of a Poincaré series is nonzero is a hard problem in general. For instance, Lehmer's conjecture on the non-vanishing of Ramanujan's $\tau$ function is equivalent to the assertion that $P_{m,12,1}$ does not vanish identically for every $m>0$.}, the Fourier coefficients of Poincaré series of \emph{negative} index, $a_n(P_{-m,k,N})$, can satisfy remarkable rationality properties.

For instance, a classical result of Petersson \cite{petersson32} and Rademacher \cite{rademacher38} yields
$$
P_{-1,2,1} = -Dj = q^{-1} - 196884 q - 42987520 q^2 - 2592899910 q^3 - \cdots \text{,}
$$
where $j$ denotes Klein's modular invariant and $D$ is the derivation $\frac{1}{2\pi i}\frac{d}{d\tau} = q\frac{d}{dq}$. Numerical experimentation provides other examples of Poincaré series with integral Fourier coefficients, such as
$$
P_{-2,6,4} = q^{-2} -35q^2 + 4096q^4  - 97686q^6 +  \cdots
$$
and
$$
P_{-1,4,9} = q^{-1} + 2q^2 -49q^5 + 48q^8 + 771q^{11} - \cdots\text{,}
$$
but most Poincaré series with negative index are also expected to have transcendental Fourier coefficients.

The present work is based on the observation that all of the above rationality phenomena can be `explained' by cohomology. The case of $P_{-1,4,9}$ was proved by Bruinier et al. \cite{BOR08} using the theory of harmonic Maass forms, the key ingredient being the fact that the unique normalised newform in $S_{4}(\Gamma_0(9))$ has CM. In the end of \cite{candelori14}, Candelori briefly discusses how to reformulate this result in terms of a certain cohomology group with local coefficients over the modular curve $X_1(9)$.

This hints to the fact that Fourier coefficients of Poincaré series, in general, are cohomological invariants. In this article, we show that $a_n(P_{m,k,N})$, $m\in \mathbb{Z}\setminus\{0\}$, are all rational expressions in \emph{periods} of the modular motive $H^1_{\cusp}(\mathcal{Y}_0(N), V_k)$ (see Sect. \ref{subsec:intro-svp} below).  These periods, in turn, are related to the classical  periods $\omega^+$ and $\omega^-$ of cusp forms, and also to their `quasi-periods' $\eta^+$ and $\eta^-$, which have recently been introduced by Brown and Hain \cite{BH18}. As we shall explain below, our result is actually more precise: it identifies the $\mathbb{Q}$-extension generated by all the Fourier coefficients $a_n(P_{m,k,N})$, for $m \in \mathbb{Z}\setminus\{0\}$, with the $\mathbb{Q}$-extension generated by a certain subclass of periods of $H^1_{\cusp}(\mathcal{Y}_0(N), \mathcal{V}_k)$, the \emph{single-valued periods}, in the terminology of Brown \cite{brown13}, \cite{BD19}. A first manifestation of identities relating Fourier coefficients of Poincaré series and single-valued periods appears already in the work of Brown \cite{brown18}. Subsequent work of Broadhurst and Acres and Broadhurst \cite{AB19} provided many other explicit examples of such identities, under the guise of `Rademacher sums' (see below). Our work can also be regarded as a generalisation of Brown's and Acres--Broadhurst's formulas.

\subsection{Single-valued periods of modular motives} \label{subsec:intro-svp}

Given a smooth algebraic variety $X$ defined over a number field $K\subset \mathbb{C}$ and an integer $i\ge 0$, there is a canonical comparison isomorphism \cite{grothendieck66}
$$
\comp : H_{\dR}^i(X)\tensor_K \mathbb{C} \stackrel{\sim}{\to} H_{\B}^i(X)\tensor_{\mathbb{Q}}\mathbb{C}
$$
given by `integration of differential forms', where $H_{\dR}^i(X)$ denotes the $i$th algebraic de Rham cohomology with coefficients in $K$ and $H^i_{\B}(X)$ denotes the $i$th Betti cohomology of $X$, defined as the dual $\mathbb{Q}$-vector space of the $i$th singular homology $H_i(X(\mathbb{C});\mathbb{Q}) = H_i(X(\mathbb{C});\mathbb{Z})\tensor_{\mathbb{Z}}\mathbb{Q}$ of the complex manifold $X(\mathbb{C})$. A \emph{period} \cite{KZ01} of $H^i(X)$ is a complex number of the form
$$
\langle \sigma,\comp(\alpha) \rangle = \int_{\sigma}\alpha \in \mathbb{C}
$$
for some $\alpha \in H^i_{\dR}(X)$ and $\sigma \in H_i(X(\mathbb{C});\mathbb{Q}) = H^i_{\B}(X)^{\vee}$.

When $K\subset \mathbb{R}$, the continuous involution $X(\mathbb{C}) \to X(\mathbb{C})$ given by complex conjugation induces a $\mathbb{Q}$-linear involution
$$
F_{\infty}: H^i_{\B}(X) \to H^i_{\B}(X)\text{,}
$$
sometimes called the `real Frobenius'. By transport of structures, we  obtain an involution $\sv \defeq \comp^{-1} \circ (F_{\infty}\tensor \id) \circ \comp$ on $H^i_{\dR}(X)\tensor_K\mathbb{C}$, which can actually be shown to be defined over $\mathbb{R}$; this is the \emph{single-valued involution}:
$$
\sv : H^i_{\dR}(X)\tensor_K \mathbb{R} \to H^i_{\dR}(X)\tensor_K \mathbb{R}\text{.}
$$
A \emph{single-valued period} \cite{BD19} of $H^i(X)$ is a real number of the form
$$
\langle \varphi, \sv(\alpha)\rangle \in \mathbb{R}
$$
for some $\alpha \in H^i_{\dR}(X)$ and $\varphi \in H^i_{\dR}(X)^{\vee}$.

Single-valued periods can be understood in terms of periods as follows. Given a period matrix $P$ of $H^i(X)$, that is, a matrix of $\comp$ in a $K$-basis of $H_{\dR}^i(X)$ and a $\mathbb{Q}$-basis of $H_{\B}^i(X)$, a single-valued period matrix is given by
\begin{equation}\label{eq:intro-sv-matrix}
S = P^{-1}\overline{P} = P^{-1}F_{\infty}P\text{.}
\end{equation}
The terminology `single-valued period' has its origin in the theory of period functions (i.e., periods of families of algebraic varieties); we refer to \cite{BD19} for further details and motivation.

\begin{ex}\label{ex:intro-X_0(11)}
  Consider the elliptic curve $E$ over $\mathbb{Q}$ given by the equation
  $$
  y^2 + y = x^3 - x^2 -10x -20\text{.}
  $$
  A basis for its first de Rham cohomology $H^1_{\dR}(E)$ is represented by $(\omega,\eta) = (\frac{dx}{2y +1}, x\frac{dx}{2y+1})$. Given a oriented basis $(\gamma_1,\gamma_2)$ of $H_1(E(\mathbb{C});\mathbb{Z})$, we obtain a period matrix
  $$
P = \left(\begin{array}{cc}
             \omega_1 & \eta_1 \\
             \omega_2 & \eta_2
      \end{array}\right) = \left(\begin{array}{cc}
             \int_{\gamma_1}\omega & \int_{\gamma_1}\eta \\
             \int_{\gamma_2}\omega & \int_{\gamma_2}\eta
      \end{array}\right) \in {\GL}_2(\mathbb{C})
    $$
    satisfying $\det P = 2\pi i$ (Legendre's period relation). The single-valued period matrix with respect to the basis $([\omega],[\eta])$ is then given by
 $$
S = \frac{1}{2\pi i}\left(\begin{array}{cc}
             \overline{\omega}_1\eta_2-\overline{\omega}_2\eta_1 & \overline{\eta}_1\eta_2 - \eta_1\overline{\eta}_2 \\
             \omega_1\overline{\omega}_2 - \overline{\omega}_1\omega_2 & \omega_1\overline{\eta}_2-\omega_2\overline{\eta}_1
      \end{array} \right) \in {\GL}_2(\mathbb{R})\text{.}
 $$
 It does not depend on the choice of $(\gamma_1,\gamma_2)$, and necessarily satisfies $S^2=\id$ and $\text{Tr }S = 0$. Numerically, we can compute a period lattice basis $\omega_1,\omega_2$ in \texttt{Sage} \cite{sage} and obtain the quasi-periods $\eta_1,\eta_2$ by integrating the Weierstrass $\wp$ function (here, $(\omega,\eta) = (dz, \wp(z)dz + \frac{1}{3}dz)$). We get
 $$
P = \left(\begin{array}{cc}
             1.269209... & -2.214333... \\
             0.634604... + i 1.458816... & -1.107166... + i 2.405338...
      \end{array}\right)
    $$
    and
    $$
S = \left(\begin{array}{cc}
             -0.028238... & -1.695389... \\
             -0.589364...  & 0.028238...
      \end{array}\right)\text{.}
    $$
\end{ex}

More generally, we can talk about periods and single-valued periods of motives, which can be thought as compatible systems of realisations (de Rham, Betti, étale, etc.) of `geometric origin'. For instance, the classical periods $\omega^+$, $\omega^-$ of a normalised Hecke cusp form $f \in S_{k+2}(\Gamma_0(1))$ (\cite{manin73}, \cite{CS17} Chapter 11) are defined as real numbers for which there is a decomposition
$$
\int_{0}^{i\infty}f(\tau)(X-\tau Y)^kd\tau = \omega^+ P^+ + i \omega^-P^-\text{, }\qquad P^{\pm} \in K_f[X,Y]\text{,}
$$
with $P^+$ (resp. $P^-$) an even (resp. odd) homogeneous polynomial in $X,Y$. Here, $K_f=\mathbb{Q}(a_n(f)\text{ ; } n \ge 1)$ is the totally real number field generated by the Fourier coefficients of $f$. Up to a power of $2\pi i$, the numbers $\omega^+$ and $\omega^-$ are periods of a pure motive $H_f$, occurring in the cohomology of a Kuga--Sato variety \cite{scholl90}.

In practice, instead of going to the Kuga--Sato variety, it is convenient to stay at the modular curve and to work with cohomology with local coefficients. For every integer $N\ge 1$ we denote by $\mathcal{Y}_0(N)$ the moduli stack over $\mathbb{Q}$ classifying elliptic curves endowed with a cyclic subgroup of order $N$. Its analytification is isomorphic to the orbifold quotient $\Gamma_0(N)\backslash \! \backslash \mathbb{H}$, and its corresponding coarse moduli space is isomorphic to the usual open modular curve $Y_0(N)$ over $\mathbb{Q}$. Let $\mathcal{E}$ be the universal elliptic curve over $\mathcal{Y}_0(N)$.

Explicitly, there is a (mixed) motive $H^1(\mathcal{Y}_0(N),V_k)$ over $\mathbb{Q}$ whose de Rham realisation is given by
$$
H^1(\mathcal{Y}_0(N),V_k)_{\dR} = H^1_{\dR}(\mathcal{Y}_0(N), \mathcal{V}_k)\text{,}
$$
where $\mathcal{V}_k$ is the vector bundle $\Sym^k H_{\dR}^1(\mathcal{E}/\mathcal{Y}_0(N))$ endowed with the Gauss--Manin connection. This cohomology group admits a description in terms of weakly holomorphic modular forms (cf. \cite{scholl85}, \cite{coleman96}, \cite{KS16}, \cite{candelori14}, \cite{BH18}): there is a canonical exact sequence of $\mathbb{Q}$-vector spaces
  \begin{equation}\label{eq:intro-ex-seq}
0 \to M_{-k}^!(\Gamma_0(N); \mathbb{Q}) \stackrel{D^{k+1}}{\to} M_{k+2}^!(\Gamma_0(N);\mathbb{Q}) \to H^1_{\dR}(\mathcal{Y}_0(N), \mathcal{V}_k) \to 0\text{,}
  \end{equation}
  where $M^!_{r}(\Gamma_0(N);\mathbb{Q})$ denotes the space of weakly holomorphic modular forms of weight $r$ and level $\Gamma_0(N)$ with rational Fourier coefficients at infinity, and $D^{k+1}$ is the `Bol operator'. The Betti realisation of $H^1(\mathcal{Y}_0(N),V_k)$ is isomorphic to the classical group cohomology of $\Gamma_0(N)$ considered by Eichler and Shimura (see Sect. \ref{sec:real-modular-motives} below).

 Motives of cusp forms are found in the weight $k+1$ part of $H^1(\mathcal{Y}_0(N),V_k)$, which is a polarisable pure motive $H_{\cusp}^1(\mathcal{Y}_0(N),V_k)$ of Hodge type $\{(k+1,0),(0,k+1)\}$ whose realisations are given by cuspidal (or parabolic) cohomology. The single-valued periods of $H_{\cusp}^1(\mathcal{Y}_0(N),V_k)$ are the main object of study of this paper. These include, up to a power of $2\pi i$, the Petersson norms of modular forms of weight $k+2$ and level $\Gamma_0(N)$, but also some non-standard quantities, as we shall see below. 

\subsection{Single-valued periods and coefficients of Poincaré series}

Our main result is the following.

\begin{theorem}[cf. Theorem \ref{thm:main} below]\label{thm:intro-main}
  Let $k\ge 0$ and $N\ge 1$ be integers, with $k$ even, and let $S = (s_{ij})_{1\le i,j\le r}$ be a  single-valued period matrix with respect to a $\mathbb{Q}$-basis of  $H_{\cusp}^1(\mathcal{Y}_0(N),V_k)_{\dR}$. Then,
  $$
\mathbb{Q}(s_{ij} \text{ ; } 1\le i,j\le r) = \mathbb{Q}(a_n(P_{m,k+2,N}) \text{ ;  for all } n\ge 1\text{ and } m\neq 0)\text{.}
  $$
\end{theorem}

As an immediate corollary of the above theorem, we deduce that the field generated by Fourier coefficients of Poincaré series of positive and negative index $\mathbb{Q}(a_n(P_{m,k+2,N}) \text{ ; } n\ge 1\text{, } m\neq 0)$ is finitely generated over $\mathbb{Q}$. In the degenerate case where $S_{k+2}(\Gamma_0(N))=0$ (or, equivalently, $H^1_{\cusp}(\mathcal{Y}_0(N),V_k) = 0$), there are no single-valued periods; in other words, $S$ should be taken as the empty matrix, and the statement amounts to asserting that the Fourier coefficients of the corresponding Poincaré series are all rational. This is the case of $P_{-1,2,1}$ considered above.

Apart from terminology, that Fourier coefficients of Poincaré series of \emph{positive} index $m$ can be written as rational expressions in single-valued periods of $H_{\cusp}^1(\mathcal{Y}_0(N),V_k)$ is in fact a classical statement following from Petersson's formula: for every $f \in S_{k+2}(\Gamma_0(N))$,
\begin{equation}\label{eq:intro-petersson-formula}
(f,P_{m,k+2,N})_{\Pet} \defeq \int_{\Gamma_0(N)\backslash\mathbb{H}}f(x+iy)\overline{P_{m,k+2,N}(x+iy)}y^{k+2}\frac{dx\wedge dy}{y^2} = \frac{k!}{(4\pi m)^{k+1}}a_m(f)\text{.}
\end{equation}
The main novelty in Theorem \ref{thm:intro-main} is the case of negative index.

\begin{ex}
  The elliptic curve $E$ of Example \ref{ex:intro-X_0(11)} is isomorphic to the modular curve $X_0(11)$. In particular, since $H^1(X_0(11)) = H^1_{\cusp}(\mathcal{Y}_0(11),V_0)$, Theorem \ref{thm:intro-main} implies that the Fourier coefficients of $P_{m,2,11}$ are rational expressions in the single-valued periods of $H^1(E)$ for every $m \in \mathbb{Z}\setminus\{0\}$. For instance, we have
$$
a_1(P_{1,2,11}) =-\frac{2\pi i}{\omega_1\overline{\omega}_2 - \overline{\omega}_1\omega_2} = 1.696742...
$$
and
$$
a_{1}(P_{-1,2,11}) = \frac{\overline{\omega}_1\eta_2-\overline{\omega}_2\eta_1}{\omega_1\overline{\omega}_2 - \overline{\omega}_1\omega_2} - 1 = -0.952086...
$$
This can be checked numerically using formulas (\ref{eq:intro-Fourier-positive}) and (\ref{eq:intro-Fourier-negative}). The expression for $a_1(P_{1,2,11})$ is classical: it follows from (\ref{eq:intro-petersson-formula}), which shows that $a_1(P_{1,2,11})^{-1} = 4\pi (f,f)_{\Pet}$, with $f$ the unique normalised cusp form in $S_2(\Gamma_0(11))$.
\end{ex}

In general, when $\dim S_{k+2}(\Gamma_0(N))=1$ (see tables in Sect. \ref{sec:rank2} below), that is, when the motive $H^1_{\cusp}(\mathcal{Y}_0(N),V_k)$ is of rank 2, we prove a refined version of Theorem \ref{thm:intro-main}. In this case, let us denote by 
$$
\left(\begin{array}{cc}
            s_{11} & s_{12} \\
            s_{21} & s_{22}
          \end{array}\right)
$$
the single-valued period matrix with respect to a symplectic basis $([f],[g])$ of $H^1_{\cusp}(\mathcal{Y}_0(N),V_k)_{\dR}$, induced by a cusp form $f$ and a weakly holomorphic modular form $g$ (see (\ref{eq:intro-ex-seq})).

\begin{theorem}[cf. Proposition \ref{prop:sv-periods-rank-2} below]\label{thm:intro-rank2}
  With the above notation, for every integer $m\ge 1$, there exists $h_m \in M_{-k}^!(\Gamma_0(N);\mathbb{Q})$, depending on $f$ and $g$, such that, for every $n\ge 1$,
   $$
a_n(P_{m,k+2,N}) = -\frac{k!}{m^{k+1}}a_m(f)a_n(f)\frac{1}{s_{21}}\ \ \text{ and }\ \ a_n(P_{-m,k+2,N}) =  \frac{k!}{m^{k+1}}a_m(f)a_n(f)\frac{s_{11}}{s_{21}} + r_{m,n} \text{,}
$$
where
$$
r_{m,n} = \frac{k!}{m^{k+1}}a_m(f)a_n(g) + n^{k+1}a_n(h_m) \in \mathbb{Q}\text{.}
$$
\end{theorem}

From the relations $S^2 = \id$ and $\text{Tr }S = 0$, we conclude that the entries $s_{12}$ and $s_{22}$ of $S$ are determined by $s_{11}$ and $s_{21}$, so that the above result indeed refines Theorem \ref{thm:intro-main} in the rank 2 case.

The above theorem also explains some rationality phenomena. For instance, when $a_m(f)=0$, then the above formulas imply that $P_{m,k+2,N}$ vanishes identically and that $a_n(P_{-m,k+2,N})$ is a rational number for all $n\ge 1$. When the unique normalised cusp form in $S_{k+2}(\Gamma_0(N))$ has CM, the motive $H^1_{\cusp}(\mathcal{Y}_0(N),V_k)$ acquires extra endomorphisms after base change to a CM field (\cite{schappacher88} Chapter V), and we can deduce from this that the quotient of single-valued periods $s_{11}/s_{21}$ is a rational number (cf. Example \ref{ex:endomorphism} below); thus, $P_{-m,k+2,N}$ has rational Fourier coefficients. 

Theorem \ref{thm:intro-rank2} generalises Brown's Corollary 1.4 in \cite{brown18}, concerning level 1 and weight 12.  It also provides a proof for the formulas in `genus 1' of Acres and Broadhurst \cite{AB19} found by numerical experimentation. To see this, one must remark that the `Rademacher sums' of \cite{AB19} are coefficients of Poincaré series after the action of the Fricke involution (cf. \cite{lefourn17} Proposition 5.4). In general, Theorem \ref{thm:intro-main} gives a conceptual explanation for the identities involving periods and quasi-periods of modular forms found by Broadhurst and Acres--Broadhurst (see also the formulas (\ref{eq:formula-rho}) below).

\begin{obs}
  In higher rank, after base change to an appropriate number field $K$, we can split $H^1_{\cusp}(\mathcal{Y}_0(N),V_k)$ by using Hecke operators; this leads to simple formulas (with coefficients in $K$) akin to the rank 2 case. This is worked out in Sect. \ref{sec:hecke} in the case where there are no old forms, such as in level 1.
\end{obs}

Our proof method relies on the theory of harmonic Maass forms and in its relation with the single-valued involution as explained by Brown \cite{brown18}. We summarise this relation in the following theorem, which reformulates and generalises Brown's results (cf. Candelori \cite{candelori14}).

\begin{theorem}[cf. Corollary \ref{coro:maass-forms} below]\label{thm:intro-maass}
The following diagram of $\mathbb{C}$-vector spaces commutes:
  $$
  \begin{tikzcd}[column sep=small]
    & H_{-k}^!(\Gamma_0(N))\arrow{ld}[swap]{\frac{1}{(4\pi)^{k+1}}\xi_{-k}}\arrow{rd}{\frac{1}{k!}D^{k+1}} &\\
    M_{k+2}^!(\Gamma_0(N)) \arrow{d} &  & M_{k+2}^!(\Gamma_0(N)) \arrow{d}\\
    H^1_{\dR}(\mathcal{Y}_0(N),\mathcal{V}_k)\tensor_{\mathbb{Q}} \mathbb{C} \arrow{rr}[swap]{\sv\tensor c_{\dR}} & & H^1_{\dR}(\mathcal{Y}_0(N),\mathcal{V}_k)\tensor_{\mathbb{Q}} \mathbb{C}
  \end{tikzcd}
  $$
  where the vertical maps are induced by (\ref{eq:intro-ex-seq}), and $c_{\dR}$ denotes complex conjugation on coefficients.
\end{theorem}

Here, $H_{-k}^!(\Gamma_0(N))$ denotes the space of harmonic Maass forms of manageable growth of weight $-k$ and level $\Gamma_0(N)$ (see Sect. \ref{sec:maass} below), and $\xi_{-k}$ and $D^{k+1}$ are the differential operators considered in \cite{BOR08}. Let us remark that Theorem \ref{thm:intro-maass} clarifies a number of constructions in the theory of weakly holomorphic modular forms and harmonic Maass forms. For instance, using the above diagram one can show that the `regularised Petersson inner product' of \cite{BDE17} is induced by the usual Hermitian pairing on $H^1_{\dR}(\mathcal{Y}_0(N),\mathcal{V}_k)\tensor_{\mathbb{Q}} \mathbb{C}$ given by the polarisation of its Hodge structure (see Remark \ref{obs:regularized} below); this immediately implies, among other properties, the compatibility of the regularised inner product with Hecke operators.  

By combining Theorem \ref{thm:intro-maass} with a theorem of Bringmann and Ono on the harmonic lifts of Poincaré series \cite{BO07}, we deduce the following result, which is the main ingredient in the proof of Theorems \ref{thm:intro-main} and \ref{thm:intro-rank2}.

\begin{prop}[cf. Proposition \ref{prop:poincare-betti-conjugate} below]
  For every integer $m\neq 0$, the image of $P_{m,k+2,N}$ in $H^1_{\cusp}(\mathcal{Y}_0(N),V_k)_{\dR}\tensor_{\mathbb{Q}} \mathbb{R}$ satisfies
  $$
\sv([P_{m,k+2,N}]) = - [P_{-m,k+2,N}]\text{.}
  $$
\end{prop}

\subsection{Organisation of the article}

In Sects. 1 and 2 we outline the basic formalism concerning single-valued periods. As in \cite{BD19}, we avoid explicit mention of a category of motives by working in a suitable category of realisations $\mathcal{H}$, which is sufficient for our purposes. We then focus on the case of pure polarisable objects of $\mathcal{H}$ of Hodge type $\{(n,0),(0,n)\}$, as `motives of modular forms' are of this type. Our main result here is Proposition \ref{prop:svp-matrix-polarisation}, which explicitly describes the algebraic relations the single-valued period matrix of such an object must satisfy.

Sections 3 and 4 concern the theory of weakly holomorphic modular forms; their main purpose is to set up notation. We first introduce several spaces of weakly holomorphic modular forms for the group $\Gamma_0(N)$, especially $S_{r}^{!,\infty}(\Gamma_0(N))$, the space of weakly holomorphic cusp forms which are holomorphic at every cusp different from $\infty$. Then, we recall their geometric interpretation in terms of the moduli stack $\mathcal{Y}_0(N)$, and the $q$-expansion principle. A key result is Lemma \ref{lemma:principal-part}, which will allow us to control the field of definition of certain Fourier coefficients.

We then proceed to an explicit description of the objects $H_{\cusp}^1(\mathcal{Y}_0(N),V_k)$ of $\mathcal{H}$ in Sects. 5 and 6. This is well known to experts, but we could not find a reference with all the properties, in the particular setting and degree of generality we need, spelled out. After describing the cuspidal de Rham cohomology group $H_{\dR,\cusp}^1(\mathcal{Y}_0(N),\mathcal{V}_k)$ (see Corollary \ref{coro:derahm-whmf}), we show how to express it as the de Rham realisation of an object in $\mathcal{H}$ (see Theorem \ref{thm:object-H}); in particular, it is endowed with a single-valued involution.

In Section 7, we `compute' the single-valued involution of $H_{\cusp}^1(\mathcal{Y}_0(N),V_k)$ in terms of harmonic lifts of modular forms in Theorem \ref{thm:maass-forms}. Using a theorem of Bringmann and Ono, this allows us to understand the action of the single-valued involution on the cohomology classes defined by Poincaré series (Proposition \ref{prop:poincare-betti-conjugate}).

Finally, in Sections 8, 9 and 10, we employ the previous results to prove our main theorems relating the single-valued periods of $H_{\cusp}^1(\mathcal{Y}_0(N),V_k)$ and the Fourier coefficients at infinity of the Poincaré series $P_{m,k+2,N}$; see Proposition \ref{prop:sv-periods-rank-2}, Theorem \ref{thm:main}, and Theorem \ref{thm:Pm-hecke-basis}.

\subsection{Terminology and notation}

Our notations for modular forms are standard. The Poincaré upper half-plane is denoted by $\mathbb{H} = \{\tau \in \mathbb{C} \mid \Im \tau >0\}$, and the left action of $\SL_2(\mathbb{R})$ on $\mathbb{H}\cup \mathbb{P}^1(\mathbb{R})$ by
$$
g\cdot \tau = \frac{a\tau +b}{c\tau+d}\text{, }\qquad \text{ where } g = \left(\begin{array}{cc}a & b \\ c & d\end{array}\right)\text{.}
$$
If $f: \mathbb{H} \to \mathbb{C}$ is a function, $g \in \SL_2(\mathbb{R})$, and $r$ is an integer, we set $j(g,\tau)=c\tau+d$, and we denote
$$
f|_{g,r}: \mathbb{H} \to \mathbb{C}\text{, }\qquad \tau \mapsto j(g,\tau)^{-r}f(\tau)\text{.}
$$

If $N\ge 1$ is an integer, the Hecke congruence subgroup of level $N$ is the subgroup of $\SL_2(\mathbb{Z})$ defined by
$$
\Gamma_0(N) = \left\{\left(\begin{array}{cc} a & b \\ c & d\end{array} \right) \in {\SL}_2(\mathbb{Z})\text{ ; } c \equiv 0 \mod N\right\}\text{.}
$$
The set of cusps of $\Gamma_0(N)$ is the quotient $\Gamma_0(N)\backslash\mathbb{P}^1(\mathbb{Q})$. The stabiliser of a cusp $p$ is denoted by $\Gamma_0(N)_p$. For $g = \left(\begin{smallmatrix}a & b \\ c & d\end{smallmatrix}\right) \in \SL_2(\mathbb{Z})$, the cusp determined by $g$ is the class of $(a:c)$ in $\Gamma_0(N)\backslash\mathbb{P}^1(\mathbb{Q})$.

We denote
$$
D = \frac{1}{2\pi i}\frac{d}{d\tau} = q\frac{d}{dq}\text{,}
$$
where $q= e^{2\pi i \tau}$.

\subsection{Acknowledgements}

This project has received funding from the European Research Council (ERC) under the European Union’s Horizon 2020 research and innovation programme (Grant Agreement No. 724638). I am greatly indebted to Francis Brown for his encouragement and his many suggestions. I also thank Jan Vonk for our discussions concerning harmonic Maass forms and Netan Dogra for his kind clarifications on the theory of modular forms.

\section{The basic formalism of single-valued periods}

Let $K$ be a subfield of $\mathbb{R}$, and $\mathcal{H}(K)$ the `category of realisations' considered in \cite{BD19} (see also \cite{deligne79}, \cite{deligne89}). Its objects are given by triples
$$
H = ((H_{\B}, W^{\B},F_{\infty}),(H_{\dR},W^{\dR},F_{\dR}),\comp)\text{,}
$$
where
\begin{itemize}
\item $H_{\B}$ is a finite-dimensional $\mathbb{Q}$-vector space with an increasing (weight) filtration $W^{\B}$ and an involution $F_{\infty}: H_{\B} \to H_{\B}$ (`real Frobenius'),
\item $H_{\dR}$ is a finite-dimensional $K$-vector space with an increasing (weight) filtration $W^{\dR}$ and a decreasing (Hodge) filtration $F_{\dR}$, and
 \item $\comp: H_{\dR}\tensor_{K} \mathbb{C} \stackrel{\sim}{\to} H_{\B}\tensor_{\mathbb{Q}}\mathbb{C}$ is a $\mathbb{C}$-linear isomorphism
\end{itemize}
such that
\begin{itemize}
\item $\comp(W^{\dR}\tensor_{K}\mathbb{C}) = W^{\B}\tensor_{\mathbb{Q}}\mathbb{C}$, 
\item $(H_B, W^{\B}, \comp(F_{\dR}))$ is a $\mathbb{Q}$-mixed Hodge structure (\cite{deligne71} 2.3.8), and
\item the diagram of $\mathbb{C}$-vector spaces
  \begin{equation}\label{diagram:comp-frobenius}
    \begin{tikzcd}
      H_{\dR}\tensor_{K}\mathbb{C} \arrow{r}{\comp} \arrow{d}[swap]{\id \tensor c_{\dR}} & H_{\B} \tensor_{\mathbb{Q}}\mathbb{C}\arrow{d}{F_{\infty}\tensor c_{\B}}\\
      H_{\dR}\tensor_{K}\mathbb{C} \arrow{r}[swap]{\comp} & H_{\B} \tensor_{\mathbb{Q}}\mathbb{C}
    \end{tikzcd}
  \end{equation}
  commutes, where $c_{\B}$ and $c_{\dR}$ denote the action of complex conjugation on coefficients.
\end{itemize}
A morphism $\varphi: H \to H'$ is a pair of maps $\varphi_{\B}: H_{\B} \to H'_{\B}$ ($\mathbb{Q}$-linear), $\varphi_{\dR}: H_{\dR} \to H'_{\dR}$ ($K$-linear), `preserving' all of the above structures; in particular, $\varphi_B$ is a morphism of $\mathbb{Q}$-mixed Hodge structures $(H_{\B},W^{\B},\comp(F_{\dR})) \to (H'_{\B},{W^{\B}}', \comp(F_{\dR}'))$.

\begin{ex}[Tate objects]\label{ex:tate-obj}
For every $n \in \mathbb{Z}$, we define an object $\mathbb{Q}(-n)$ of $\mathcal{H}(K)$ as follows:  $\mathbb{Q}(-n)_{\B} = \mathbb{Q}$, $W^{\B}_{2n} = \mathbb{Q}(-n)_{\B}$, $W^{\B}_{2n-1} = 0$, $F_{\infty}=(-1)^n\id$, $\mathbb{Q}(-n)_{\dR} =  K$, $W^{\dR}_{2n} = \mathbb{Q}(-n)_{\dR}$, $W^{\dR}_{2n-1} = 0$,  $F_{\dR}^n = \mathbb{Q}(-n)_{\dR}$, $F_{\dR}^{n+1}=0$, and $\comp: z\mapsto (2\pi i)^nz$.
\end{ex}

The commutativity of the diagram (\ref{diagram:comp-frobenius}) implies that the $\mathbb{C}$-linear involution \linebreak $\comp^{-1}\circ (F_{\infty}\tensor \id) \circ \comp$ of $H_{\dR}\tensor_{K}\mathbb{C}$ commutes with $\id \tensor c_{\dR}$, so that there exists a unique $\mathbb{R}$-linear involution, the \emph{single-valued involution}
$$
\sv: H_{\dR}\tensor_{K}\mathbb{R} \to H_{\dR}\tensor_{K}\mathbb{R}\text{,}
$$
whose $\mathbb{C}$-linear extension to $H_{\dR}\tensor_{K}\mathbb{C} = (H_{\dR}\tensor_{K}\mathbb{R})\tensor_{\mathbb{R}}\mathbb{C}$ is
\begin{align}\label{eq:sv}
\sv\tensor \id = \comp^{-1}\circ (F_{\infty}\tensor \id) \circ \comp\text{.}
\end{align}

Note that a morphism $\varphi: H \to H'$ in $\mathcal{H}(K)$ commutes with the single-valued involutions:
\begin{align}\label{eq:sv-commute}
\sv \circ (\varphi_{\dR}\tensor \id) = (\varphi_{\dR}\tensor \id)\circ \sv\text{,}
\end{align}
where $\varphi_{\dR}\tensor \id : H_{\dR}\tensor_{K}\mathbb{R} \to H_{\dR}'\tensor_{K}\mathbb{R}$ is the $\mathbb{R}$-linear extension of $\varphi_{\dR}$.

We can express the action of the Betti complex conjugation  $\id \tensor c_{\B}$  on the de Rham side in terms of the single-valued involution as follows.

\begin{lemma}\label{lemma:betti-conj}
  For any object $H$ of $\mathcal{H}(K)$, we have
  $$
\sv\tensor c_{\dR} = \comp^{-1}\circ (\id \tensor c_{\B}) \circ \comp
$$
on $H_{\dR}\tensor_{K}\mathbb{C}$.
\end{lemma}

\begin{proof}
 We have:
  \begin{align*}   
    \text{sv} \tensor c_{\dR} & = (\text{sv} \tensor \text{id})\circ (\text{id} \tensor c_{\dR}) && \\
                              & = \text{comp}^{-1}\circ (F_{\infty}\tensor \text{id}) \circ \text{comp}\circ (\text{id} \tensor c_{\dR}) &&  \text{by definition of }\text{sv} \text{ (\ref{eq:sv})}\\
                              & = \comp^{-1} \circ (F_{\infty}\tensor \id) \circ (F_{\infty}\tensor c_{\B}) \circ \comp && \text{by the commutativity of (\ref{diagram:comp-frobenius})}\\
      &= \comp^{-1}\circ (\id \tensor c_{\B}) \circ \comp && \text{since }F_{\infty}\text{ is an involution} 
  \end{align*} 
\end{proof}

\begin{defi}
  A \emph{single-valued period} of an object $H$ of $\mathcal{H}(K)$ is any real number of the form
  $$
\langle \varphi, \sv(\omega)\rangle \in \mathbb{R}
$$
where $\omega \in H_{\dR}$, $\varphi \in H_{\dR}^{\vee}\defeq \Hom_K(H_{\dR},K)$, and $\langle \ , \ \rangle$ denotes the natural duality pairing.
\end{defi}

If $r$ denotes the dimension of the $K$-vector space $H_{\dR}$, and 
$$
b_{\dR}: K^{\oplus r} \stackrel{\sim}{\to} H_{\dR}
$$
is a $K$-basis of $H_{\dR}$, then there exists a unique $S \in \GL_r(\mathbb{R})$, the matrix of $\sv$ in the basis $b_{\dR}$, such that
$$
\sv \circ b_{\dR} = b_{\dR} \circ S\text{.}
$$
The coefficients of $S$ are single-valued periods of $H$ and generate the $K$-linear span of all single-valued periods of $H$. Since $\sv$ is an involution,  single-valued periods always satisfy the  relations
$$
S^2 = \id\qquad \text{ and }\qquad \text{Tr } S = \text{Tr } F_{\infty}\text{.}
$$

Single-valued periods may also satisfy other relations:

\begin{ex}\label{ex:endomorphism}
  An endomorphism $\varphi$ of $H$ in $\mathcal{H}(K)$ induces a $K$-linear relation between its single-valued periods. Keeping the above notation, let $M \in \GL_r(K)$ be the matrix of $\varphi_{\dR}$ in the basis $b_{\dR}$, so that $\varphi_{\dR}\circ b_{\dR} = b_{\dR}\circ M$. By (\ref{eq:sv-commute}), and by definition of $M$ and $S$, we get
  $$
  MS = SM\text{.}
  $$
  We remark that, if $K$ is a real number field and $L\subset \mathbb{C}$ is a finite extension of $K$, then we can also define a category of realisations $\mathcal{H}(L)$ (see \cite{BD19} 2.1.2), and we get similar linear relations over $L$, with $M \in \GL_r(L)$, if $\varphi$ is now an endomorphism of $H\tensor_K L$ in $\mathcal{H}(L)$.
\end{ex}

Let us recall how the notion of single-valued periods connects with the notion of periods (cf. \cite{BD19} 2.4). The \emph{periods} of an object $H$ of $\mathcal{H}(K)$ are complex numbers of the form
  $$
  \langle \sigma, \comp(\omega) \rangle
  $$
  where $\omega \in H_{\dR}$ and $\sigma \in H_B^{\vee} \defeq \Hom_{\mathbb{Q}}(H_{\B},\mathbb{Q})$. We can write single-valued periods in terms of periods as follows. If $b_{\B}: \mathbb{Q}^{\oplus r} \stackrel{\sim}{\to} H_{\B}$ is a $\mathbb{Q}$-basis of $H_{\B}$, then the \emph{period matrix} of $\comp$ with respect to $b_{\dR}$ and $b_{\B}$ is the unique matrix $P \in \GL_r(\mathbb{C})$ satisfying
  $$
\comp \circ b_{\dR} = b_{\B} \circ P\text{.}
$$
The coefficients of $P$ generate the $K$-linear span of all the periods of $H$. Using Lemma \ref{lemma:betti-conj} and the definitions of $S$ and $P$, we get
$$
S = \overline{P}^{-1}P = P^{-1}\overline{P}\text{.}
$$
Further, if we let $R \in \GL_{r}(\mathbb{Q})$ be such that $F_{\infty}\circ b_{\B} = b_{\B} \circ R$, then $\overline{P} = RP$, and we get
$$
S = P^{-1}RP\text{.}
$$

\begin{obs}
 Beware that a period of an object $H$ of $\mathcal{H}(K)$ as defined above is only a period in the sense of Kontsevich and Zagier \cite{KZ01} if $K$ is a number field and $H$ comes from a mixed motive over $K$, e.g., $H=H^i(X)$ for some $K$-algebraic variety $X$.
\end{obs}

\section{Single-valued periods in the presence of a polarisation} \label{sec:sv-per-pol}

Let $H$ be an object of $\mathcal{H}(K)$ such that $(H_{\B},W^{\B},\comp(F_{\dR}))$ is a pure $\mathbb{Q}$-Hodge structure of weight $n \in \mathbb{Z}$. This means that $W^{\B}_n = H_{\B}$, $W^{\B}_{n-1}=0$, and
$$
H_{\B}\tensor_{\mathbb{Q}}\mathbb{C} = \bigoplus_{p+q=n} H^{p,q}
$$
where
$$
H^{p,q} := \comp(F_{\dR}^p) \cap (\id \tensor c_{\B})(\comp(F_{\dR}^q))\text{.}
$$
We say that $H$ is \emph{pure of weight $n$}; the \emph{Hodge type} of $H$ is the set $\{(p,q) \in \mathbb{Z}^2 \mid H^{p,q}\neq 0\}$.

\begin{defi}
  A \emph{polarisation} of a pure object $H$ of $\mathcal{H}(K)$ of weight $n$ is a morphism
  $$
\langle \ , \ \rangle : H\tensor H \to \mathbb{Q}(-n)
$$
in $\mathcal{H}(K)$ inducing a polarisation on the pure Hodge structure $(H_{\B},\comp(F_{\dR}))$ (\cite{deligne71} 2.1.15). This means that the $\mathbb{C}$-linear extension of $\langle \ , \ \rangle_{\B}$ to $H_{\B}\tensor_{\mathbb{Q}} \mathbb{C}$ is $(-1)^n$-symmetric and satisfies the classical `Hodge--Riemann relations':
\begin{enumerate}
\item $\langle H^{p,q},H^{p',q'}\rangle_{\B} = 0$ if $(p,q)\neq (q',p')$
  \item $i^{p-q}\langle \alpha,(\id \tensor c_{\B})(\alpha)\rangle_{\B} >0$ for every $\alpha \in H^{p,q}\setminus\{0\}$. 
\end{enumerate}
\end{defi}

 It follows from the commutativity of
  \begin{equation}\label{eq:comm-pol}
  \begin{tikzcd}
    (H_{\dR}\tensor_{K}H_{\dR})\tensor_{\mathbb{Q}}\mathbb{C} \arrow{rr}{\langle \ , \ \rangle_{\dR}\tensor \id } \arrow{d}[swap]{\comp}& & \mathbb{Q}(-n)_{\dR}\tensor_{\mathbb{Q}}\mathbb{C}\arrow{d}{\comp}\\
    (H_{\B}\tensor_{\mathbb{Q}}H_{\B})\tensor_{\mathbb{Q}}\mathbb{C} \arrow{rr}[swap]{\langle \ , \ \rangle_{\B}\tensor \id } & &\mathbb{Q}(-n)_{\B}\tensor_{\mathbb{Q}}\mathbb{C}
  \end{tikzcd}
  \end{equation}
  and from Example \ref{ex:tate-obj} that
  \begin{align}\label{eq:comp-polarisation}
  (2\pi i)^n\langle \omega, \eta \rangle_{\dR} = \langle \comp(\omega),\comp(\eta)\rangle_{\B}
  \end{align}
  for every $\omega,\eta \in H_{\dR}\tensor_{K}\mathbb{C}$.

\begin{obs}\label{obs:non-degeneracy}
The second Hodge--Riemann relation implies that the $\mathbb{Q}$-bilinear pairing $\langle \ , \ \rangle_{\B}$, and hence also the $K$-bilinear pairing $\langle \ , \ \rangle_{\dR}$ by (\ref{eq:comp-polarisation}), is non-degenerate.
\end{obs}

The de Rham pairing $\langle \ , \ \rangle_{\dR}$ is compatible with the single-valued involution as follows.

\begin{lemma}\label{lemma:compatibility-sv-polarisation}
  For every $\omega,\eta \in H_{\dR}\tensor_{K}\mathbb{C}$, we have
  $$
\langle (\sv \tensor c_{\dR})(\omega), (\sv \tensor c_{\dR})(\eta)\rangle_{\dR} = (-1)^n\overline{\langle \omega,\eta \rangle}_{\dR}\text{.}
  $$
\end{lemma}

\begin{proof}
  Since the de Rham pairing is defined over $K\subset \mathbb{R}$, it suffices to prove the above formula for $\omega,\eta \in H_{\dR}$. In this case, we have $(\sv \tensor c_{\dR})(\omega) = \sv(\omega)$ and similarly for $\eta$. The equality $\langle \sv(\omega),\sv(\eta)\rangle_{\dR} = (-1)^n\langle \omega,\eta\rangle_{\dR}$ then follows from (\ref{eq:sv-commute}) and from the fact that the single-valued involution of $\mathbb{Q}(-n)$ is $(-1)^n$ (Example \ref{ex:tate-obj}).

  Alternatively, we may apply Lemma \ref{lemma:betti-conj} together with the formula (\ref{eq:comp-polarisation}), and use that $\langle (\id \tensor c_{\B})(\alpha), (\id \tensor c_{\B})(\beta) \rangle_{\B} = \overline{\langle \alpha,\beta\rangle}_{\B}$ for every $\alpha,\beta \in H_{\B}\tensor_{\mathbb{Q}}\mathbb{C}$. 
\end{proof}

The presence of a polarisation imposes additional algebraic relations on single-valued periods. For future reference, we spell out these relations in the particular case where $H$ is pure of Hodge type $\{(n,0),(0,n)\}$ for some $n\in \mathbb{Z}$.

In this case, let $d \defeq \dim F_{\dR}^n$, so that $\dim H_{\dR} = 2d$. Consider a $K$-basis
\begin{align}\label{eq:symplectic-Hodge-basis}
b_{\dR} = (\omega_1,\ldots,\omega_d,\eta_1,\ldots,\eta_d)
\end{align}
of $H_{\dR}$ such that
\begin{itemize}
\item $(\omega_1,\ldots,\omega_d)$ is a $K$-basis of $F_{\dR}^n$;
\item $\langle \omega_i, \eta_j \rangle_{\dR} = \delta_{ij}$ and $\langle \eta_i,\eta_j \rangle_{\dR} = 0$ for every $1\le i,j\le d$.
\end{itemize}
Note that $\langle \omega_i, \omega_j \rangle_{\dR}=0$ because $F_{\dR}^n$ is isotropic for $\langle \ , \ \rangle_{\dR}$ by the first Hodge--Riemann relation. Such bases always exist by a simple variation of the Gram--Schmidt process (cf. Remark \ref{obs:non-degeneracy}).

Let us write the matrix $S$ of the single-valued involution $\sv: H_{\dR}\tensor_{K}\mathbb{R} \to H_{\dR}\tensor_{K}\mathbb{R}$ in the basis $b_{\dR}$ in block form:
$$
S= \left(\begin{array}{cc}
           A & B \\
           C & D
         \end{array}\right) \in {\GL}_{2d}(\mathbb{R})\text{,}
       $$
where $A,B,C,D \in \M_{d\times d}(\mathbb{R})$. 
       
\begin{prop}\label{prop:svp-matrix-polarisation}
 Let $H$ be a polarisable pure object of $\mathcal{H}(K)$, pure of Hodge type $\{(n,0),(0,n)\}$ for some $n\in \mathbb{Z}$. With the above notation, we have $C \in \GL_d(\mathbb{R})$, and
  $$
  B = (1-A^2)C^{-1}\text{, }\ \ \ D=(-1)^{n}A^t\text{,}\ \ \ CA = (-1)^{n+1}A^tC\text{, } \ \ \ C = C^t\text{.}
  $$
  In particular, the field $K(S)$ generated by all of the coefficients of $S$ is equal to the field $K(A,C)$, generated by the coefficients of $A$ and $C$.
\end{prop}

In coordinate-free parlance, the field of rationality of $\sv : H_{\dR}\tensor_{K} \mathbb{R} \to H_{\dR}\tensor_{K} \mathbb{R}$ coincides with the field of rationality of $\sv|_{F_{\dR}^n\tensor_K\mathbb{R}}: F_{\dR}^n\tensor_K\mathbb{R} \to  H_{\dR}\tensor_{K} \mathbb{R}$.

\begin{proof}
  Since $\omega_i$ are defined over $K\subset \mathbb{R}$, we have $(\sv \tensor c_{\dR})(\omega_i) = \sv(\omega_i)$. It follows from Lemma \ref{lemma:betti-conj} and from the Hodge decomposition $H_{\B}\tensor_{\mathbb{Q}}\mathbb{C} = H^{n,0} \oplus H^{0,n}$ that
 $$
  (\omega_1,\ldots,\omega_d,\sv(\omega_1),\ldots,\sv(\omega_d)) = (\omega_1,\ldots,\omega_d,\eta_1,\ldots,\eta_d) \cdot \left(\begin{array}{cc}                                                                                       1 & A \\
                            0 & C                                                                    \end{array}\right)
 $$
 is a basis of $H_{\dR}\tensor_K\mathbb{R}$; this proves that $C \in \GL_{d}(\mathbb{R})$.

 It follows from the definition of the basis $b_{\dR}$, and from Lemma \ref{lemma:compatibility-sv-polarisation}, that
 $$
 S^t \left(\begin{array}{cc}
             0 & (-1)^n \\
             1 & 0
     \end{array}\right)S = (-1)^n \left(\begin{array}{cc}
             0 & (-1)^n \\
             1 & 0
     \end{array}\right)
   $$
   or, equivalently, using that $S^{-1}=S$,
   $$
\left(\begin{array}{cc}
             C^t & (-1)^nA^t \\
             D^t & (-1)^nB^t
     \end{array}\right) = \left(\begin{array}{cc}
             C & D \\
             (-1)^nA & (-1)^nB
     \end{array}\right)
   $$
   This proves the relations $D=(-1)^nA^t$ and $C=C^t$.

 Finally, the equation $S^2=1$ gives
$$
S^2 = \left(\begin{array}{cc}
            A^2 +BC & AB+BD \\
            CA +DC & CB+D^2
    \end{array}\right) = \left(\begin{array}{cc}
            1 & 0 \\
            0 & 1
    \end{array}\right)
  $$
so that $B= (1-A^2)C^{-1}$, and $CA = -DC = (-1)^{n+1}A^tC$. 
\end{proof}

\begin{obs}
  With the above notation, set $R \defeq AC^{-1}$. Then, $R^t=(-1)^{n+1}R$, $C^t=C$, and we have
  $$
S = \left(\begin{array}{cc}
             RC & C^{-1} - RCR \\
             C & -CR
     \end{array}\right) \in {\GL}_{2d}(\mathbb{R})\text{.}
   $$
\end{obs}

We finish this section with an alternative way of encoding the single-valued period matrix, in terms of the Hermitian form defined by the polarisation. 

\begin{lemma}\label{lemma:derham-hermitian}
  The $\mathbb{R}$-bilinear form on $H_{\dR}\tensor_{K}\mathbb{C}$ with values in $\mathbb{C}$ defined by
    $$
( \omega,\eta)_{\dR} \defeq (-1)^n\langle \omega, (\sv\tensor c_{\dR})(\eta)\rangle_{\dR}
$$
is Hermitian. Moreover, the restriction of $( \ , \ )_{\dR}$ to $F_{\dR}^n\tensor_{K}\mathbb{C}$ is positive definite.
\end{lemma}

\begin{proof}
  That $( \ , \ )_{\dR}$ is a Hermitian form (i.e., $( z\omega , w\eta )_{\dR} = z\overline{w}(\omega,\eta)_{\dR}$ for $z,w \in \mathbb{C}$) follows from formula (\ref{eq:comp-polarisation}) and from Lemma \ref{lemma:betti-conj}. The positive definiteness on $F_{\dR}^n\tensor_{K}\mathbb{C}$ is a consequence of the second Hodge--Riemann relation.
\end{proof}

Consider a de Rham basis $b_{\dR} = (\omega_1,\ldots,\omega_d,\eta_1,\ldots,\eta_d)$ as above. Then
\begin{align}\label{eq:dR-form-sv-period}
  (\omega_i,\omega_j)_{\dR} = (-1)^n\langle \omega_i, \sv(\omega_j)\rangle_{\dR}=(-1)^n\sum_{k=1}^d\left(A_{kj}\langle \omega_i,\omega_k\rangle_{\dR} + C_{kj}\langle \omega_i,\eta_k \rangle_{\dR}\right) = (-1)^nC_{ij}\text{,}
\end{align}
and, similarly,
\begin{align}\label{eq:dR-form-sv-period2}
 (\omega_i,\eta_j)_{\dR} = (-1)^nD_{ij}\text{, }\ \ \ (\eta_i,\omega_j)_{\dR} = A_{ij}\text{, } \ \ \ (\eta_i,\eta_j)_{\dR} = B_{ij}\text{.}
\end{align}

Note that we proved in Proposition \ref{prop:svp-matrix-polarisation} that $C$ is invertible; using the above formulas, we see that this is equivalent to the positive definiteness of $( \ , \ )_{\dR}$ over $F_{\dR}^n\tensor_K\mathbb{C}$ (Lemma \ref{lemma:derham-hermitian}).

\section{Spaces of weakly holomorphic modular forms}\label{sec:spaces}

Let $r$ and $N\ge 1$ be integers.

\begin{defi}
A \emph{weakly holomorphic modular form} of weight $r$ and level $\Gamma_0(N)$ is a holomorphic function $f: \mathbb{H} \to \mathbb{C}$ which is modular of weight $r$ for $\Gamma_0(N)$:
  $$
f|_{\gamma,r} = f
$$
for every $\gamma \in \Gamma_0(N)$, and meromorphic at all cusps: for every $g \in \SL_2(\mathbb{Z})$, there exists $\rho >0$ such that
$$
f|_{g,r}= O(e^{\rho \Im \tau})
$$
as $\Im \tau \rightarrow +\infty$.
\end{defi}

The $\mathbb{C}$-vector space of weakly holomorphic modular forms of weight $r$ and level $\Gamma_0(N)$ is denoted by $M^!_r(\Gamma_0(N))$. Note that $M_r^!(\Gamma_0(N))=0$ for every odd $r$, since $\bigl( \begin{smallmatrix}-1 & 0\\ 0 & -1\end{smallmatrix}\bigr) \in \Gamma_0(N)$.

Every $f \in M_r^!(\Gamma_0(N))$ admits a \emph{Fourier expansion at infinity}, which we denote by
$$
f = \sum_{n \gg -\infty}a_n(f)q^n
$$
where $q = e^{2\pi i \tau}$, and $a_n(f) \in \mathbb{C}$. For every $n \in \mathbb{Z}$, we can regard
$$
a_n : M_r^!(\Gamma_0(N)) \to \mathbb{C}\text{, }\qquad f \mapsto a_n(f)
$$
as a linear functional on $M_r^!(\Gamma_0(N))$. If $K$ is any subfield of $\mathbb{C}$, we denote by
$$
M_r^!(\Gamma_0(N);K)
$$
the $K$-subspace of $M_r^!(\Gamma_0(N))$ consisting of those $f$ such that $a_n(f)\in K$ for every $n \in \mathbb{Z}$.

\begin{defi}
For $f \in M_r^!(\Gamma_0(N);K)$, the (extended) \emph{principal part of $f$ at the cusp at infinity} is defined by the finite sum $\mathcal{P}_f = \sum_{n\le 0}a_n(f)q^n \in K[q^{-1}]$.
\end{defi}

There are also Fourier expansions at other cusps. Let
$$
g = \left(\begin{array}{cc}
      a & b \\
      c & d
    \end{array}\right) \in {\SL}_2(\mathbb{Z})\text{.}
$$
If $w$ denotes the smallest positive integer such that $\bigl( \begin{smallmatrix}1 & w\\ 0 & 1\end{smallmatrix}\bigr) \in g^{-1}\Gamma_0(N)g$ (the `\emph{width} of the cusp determined by $g$'), then we can write
$$
f|_{g,r} = \sum_{n\gg -\infty}a_{n,g}(f) q^{\frac{n}{w}}  
$$
for some $a_{n,g}(f) \in \mathbb{C}$. This is a `Fourier expansion of $f$ at the cusp $[(a:c)] \in  \Gamma_0(N)\backslash \mathbb{P}^1(\mathbb{Q})$', but it really depends on $g$: for $j \in \mathbb{Z}$, we have
$$
a_{n,gT^j} = e^{2\pi i \frac{j}{w}}a_{n,g}\text{,}
$$
where $T= \left(\begin{smallmatrix}1 & 1 \\ 0 & 1\end{smallmatrix}\right)$. Note however that the vanishing of the $n$th Fourier coefficient at a given cusp is a well defined property.

\begin{defi}
A \emph{weakly holomorphic cusp form} of weight $r$ and level $\Gamma_0(N)$ is a weakly holomorphic modular form $f \in M_r^!(\Gamma_0(N))$ such that $a_{0,g}(f)=0$ for every $g \in \SL_2(\mathbb{Z})$.
\end{defi}

We denote the subspace of weakly holomorphic cusp forms of weight $r$ and level $\Gamma_0(N)$ by $S_r^!(\Gamma_0(N)) \subset M_r^!(\Gamma_0(N))$. Similarly, for every subfield $K\subset \mathbb{C}$, we set
$$
S_r^!(\Gamma_0(N);K) \defeq S_r^!(\Gamma_0(N))\cap M_r^!(\Gamma_0(N);K)\text{.}
$$

Note that the space $M_r(\Gamma_0(N))$ (resp. $S_r(\Gamma_0(N))$) of modular forms (resp. cusp forms) of weight $r$ and level $\Gamma_0(N)$ is the subspace of $M_r^!(\Gamma_0(N))$ consisting of those $f$ such that $a_{n,g}(f)=0$ for every $n<0$ (resp. $n\le 0$) and every $g \in \SL_2(\mathbb{Z})$. For a subfield $K$ of $\mathbb{C}$, we also set
$$
M_r(\Gamma_0(N);K) \defeq M_r(\Gamma_0(N))\cap M_r^!(\Gamma_0(N);K)
$$
and
$$
S_r(\Gamma_0(N);K) \defeq S_r(\Gamma_0(N))\cap S_r^!(\Gamma_0(N);K) \text{.}
$$

\begin{obs}\label{obs:negative-weight}
  Recall that we always have $M_r(\Gamma_0(N))=0$ for $r<0$, although $M_r^!(\Gamma_0(N))$ can be non-trivial. This follows from the `valence formula' (see \cite{CS17} 5.6.1): for every $f \in M_{r}^{!}(\Gamma_0(N))$,
  $$
\frac{1}{[\SL_2(\mathbb{Z}) : \Gamma_0(N)]}\sum_{p \in \Gamma_0(N)\backslash\mathbb{H} \cup \mathbb{P}^1(\mathbb{Q})}\frac{\ord_{p}(f)}{e_p} =  \frac{r}{12}
$$
where $e_{p}=2$ (resp. $e_{p}=3$) if $p= \Gamma_0(N)\cdot i$ (resp. $p= \Gamma_0(N)\cdot e^{\frac{2\pi i}{3} }$), and $e_{p}=1$ otherwise.
\end{obs}

We shall also consider weakly holomorphic modular forms which are holomorphic at all \emph{finite} cusps. Namely, we define a subspace $M_r^{!,\infty}(\Gamma_0(N))$ of $M^!_r(\Gamma_0(N))$ consisting of $f$ such that $a_{n,g}(f)=0$ for every $n<0$ and $g = \bigl( \begin{smallmatrix}a & b\\ c & d\end{smallmatrix}\bigr) \in \SL_2(\mathbb{Z})$ with $c\not\equiv 0 \pmod N$. In other words, $M_r^{!,\infty}(\Gamma_0(N))$ is the space of weakly holomorphic modular forms with no pole at any cusp different from $\infty$, but which are allowed to have poles at $\infty$. We also define $S_r^{!,\infty}(\Gamma_0(N)) \defeq S_r^!(\Gamma_0(N)) \cap M_{r}^{!,\infty}(\Gamma_0(N))$ and, for a subfield $K$ of $\mathbb{C}$,
$$
M_{r}^{!,\infty}(\Gamma_0(N);K) \defeq M_r^{!,\infty}(\Gamma_0(N)) \cap M_r^!(\Gamma_0(N);K)
$$
and
$$
S_{r}^{!,\infty}(\Gamma_0(N);K) \defeq S_r^{!,\infty}(\Gamma_0(N)) \cap S_r^!(\Gamma_0(N);K)\text{.}
$$
We then have the following diagram of $K$-vector spaces, where all the arrows are the natural inclusions
$$
\begin{tikzcd}
  M_r(\Gamma_0(N);K) \arrow[hook]{r} & M^{!,\infty}_r(\Gamma_0(N);K) \arrow[hook]{r} & M^!_r(\Gamma_0(N);K)\\
  S_r(\Gamma_0(N);K) \arrow[hook]{r}\arrow[hook]{u} & S^{!,\infty}_r(\Gamma_0(N);K) \arrow[hook]{u}\arrow[hook]{r} & S^!_r(\Gamma_0(N);K)\arrow[hook]{u}
\end{tikzcd}
$$
Note that $M_r(\Gamma_0(N);K)$ (resp. $S_r(\Gamma_0(N);K)$) is the subspace of $M^{!,\infty}_r(\Gamma_0(N);K)$ of weakly holomorphic modular forms with constant (resp. vanishing) principal part at the cusp at infinity. 

We conclude this subsection with the example of Poincaré series. Let $\Gamma_{\infty} \le \Gamma_0(N)$ be the stabiliser of $\infty=(1:0) \in \mathbb{P}^1(\mathbb{Q})$. The proofs of the following propositions are well known in weight $k\ge 4$; see, for instance, \cite{CS17} Chapter 8. For Poincaré series of weight 2, a precise exposition can be found in \cite{rankin77} 5.7. 

\begin{prop}\label{prop:poincare}
  Let $k\ge 2$ and $m$ be integers, with $k$ even.
  \begin{enumerate}
  \item If $k\ge 4$, then
    the series
  $$
P_{m,k,N}(\tau) = \sum_{\gamma \in \Gamma_{\infty}\backslash\Gamma_0(N)}\frac{e^{2\pi i m \gamma \cdot \tau}}{j(\gamma,\tau)^k}
$$
is uniformly absolutely convergent on sets of the form $\{\tau \in \mathbb{H} \mid |\Re \tau|\le \alpha\text{, } \Im \tau \ge \beta\}$, for $\alpha,\beta > 0$, and we have
$$
P_{m,k,N} \in S_k^{!,\infty}(\Gamma_0(N);\mathbb{R})\text{.}
$$
\item If $k=2$, then the series
  $$
P_{m,2,N}(\tau,\epsilon) = \sum_{\gamma \in \Gamma_{\infty}\backslash\Gamma_0(N)}\frac{e^{2\pi i m \gamma \cdot \tau}}{j(\gamma,\tau)^2|j(\gamma,\tau)|^{2\epsilon}}
$$
is uniformly absolutely convergent on sets of the form $\{(\tau,\epsilon) \in \mathbb{H} \times \mathbb{R} \mid |\Re \tau|\le \alpha\text{, } \Im \tau \ge \beta\text{, }\epsilon \ge \epsilon_0\}$, for $\alpha,\beta,\epsilon_0 > 0$. For every $\gamma \in \Gamma_0(N)$, we have
$$
P_{m,2,N}(\gamma \cdot \tau,\epsilon) = j(\gamma,\tau)^2|j(\gamma,\tau)|^{2\epsilon} P_{m,2,N}( \tau,\epsilon)\text{,}
$$
and, if $m\neq 0$, then
$$
P_{m,2,N}\defeq \lim_{\epsilon \rightarrow 0+}P_{m,2,N}(\cdot ,\epsilon) \in S_{2}^{!,\infty}(\Gamma_0(N);\mathbb{R})\text{.}
$$
  \end{enumerate}
\end{prop}

Recall that, for $a,b,c \in \mathbb{Z}$, with $c\ge 1$, the Kloosterman sum $K(a,b;c)$ is the real algebraic number defined by
$$
K(a,b;c) = \sum_{x \in (\mathbb{Z}/c\mathbb{Z})^{\times}}e^{\frac{2\pi i}{m}(ax + bx^{-1})}\text{.}
$$ 

\begin{prop}\label{prop:poincare-fourier}
  Let $k\ge 2$ and $m\ge 1$ be integers, with $k$ even.
\begin{enumerate}
\item We have $P_{m,k,N} \in S_k(\Gamma_0(N);\mathbb{R})$ and, for every integer $n\ge 1$,
  $$
a_n(P_{m,k,N}) = \delta_{m,n}+  2\pi(-1)^{\frac{k}{2}}\left(\frac{n}{m} \right)^{\frac{k-1}{2}}\sum_{c\ge 1\text{, }N \mid c}\frac{K(m,n;c)}{c}J_{k-1}\left(\frac{4\pi\sqrt{mn}}{c} \right)
$$
where $J_{k-1}$ is the $J$-Bessel function of order $k-1$ (see \cite{rankin77} (5.3.2)).
\item The principal part of $P_{-m,k,N}$ at the cusp at infinity is $q^{-m}$, and, for every integer $n\ge 1$,
  $$
a_n(P_{-m,k,N}) = 2\pi(-1)^{\frac{k}{2}} \left(\frac{n}{m} \right)^{\frac{k-1}{2}}\sum_{c\ge 1\text{, }N\mid c}\frac{K(-m,n;c)}{c}I_{k-1}\left(\frac{4\pi\sqrt{mn}}{c}\right)
  $$
  where $I_{k-1}$ is the $I$-Bessel function of order $k-1$ (see \cite{rankin77} (5.3.3)).
\end{enumerate}
\end{prop}

\section{Geometric interpretation; the $q$-expansion principle}\label{subsec:geom-interpr}

We work over the  moduli stacks $\mathcal{Y}_0(N)$ and $\mathcal{X}_0(N)$ over $\Spec \mathbb{Q}$. By definition, for every $\mathbb{Q}$-scheme $S$, the fibre of $\mathcal{Y}_0(N)$ at $S$ is given by pairs $(E, C)$, where $E$ is an elliptic curve over $S$, and $C$ is a cyclic $S$-subgroup scheme of $E$ of order $N$ (i.e., locally for the fppf topology on $S$, the subgroup $C$ admits a generator of order $N$; see \cite{KM85} 3.4). The `compactified' moduli stack $\mathcal{X}_0(N)$ has a similar definition, but now $E$ is allowed to be a generalised elliptic curve, as in \cite{DR73}.

One can prove that $\mathcal{X}_0(N)$ is a proper smooth Deligne--Mumford stack over $\Spec \mathbb{Q}$ containing $\mathcal{Y}_0(N)$ as an open substack. More precisely, it follows from \cite{DR73}, Théorème 3.4, that  $\mathcal{Y}_0(N)$ is the complement of a closed substack $\mathcal{Z}_N\subset \mathcal{X}_0(N)$, the \emph{cuspidal locus of} $\mathcal{X}_0(N)$, which is finite étale over $\Spec \mathbb{Q}$ (therefore representable by a $\mathbb{Q}$-scheme).

\begin{obs}
The moduli stacks $\mathcal{Y}_0(1)$ and $\mathcal{X}_0(1)$ are often denoted by $\mathcal{M}_{1,1}$ and $\overline{\mathcal{M}}_{1,1}$ in the literature.
\end{obs}

In the category of complex analytic spaces, let $\mathbb{E}$ be the elliptic curve over $\mathbb{H}$ whose fibre at $\tau \in \mathbb{H}$ is the complex torus
$$
\mathbb{E}_{\tau} = \mathbb{C}/(\mathbb{Z} + \tau \mathbb{Z})\text{.}
$$ 
Then the map
\begin{align}\label{eq:unif-Y_0(N)}
  \mathbb{H} \to \mathcal{Y}_0(N)_{\mathbb{C}}^{\an}\text{, } \qquad \tau \mapsto (\mathbb{E}_{\tau},(\tfrac{1}{N}\mathbb{Z} + \tau \mathbb{Z})/(\mathbb{Z} + \tau \mathbb{Z}))
\end{align}
induces an isomorphism
$$
\mathcal{Y}_0(N)_{\mathbb{C}}^{\an} \cong \Gamma_0(N)\backslash\! \backslash\mathbb{H}\text{,}
$$
where the double backslash stands for the `orbifold quotient' (or `stacky quotient') of $\mathbb{H}$ by the left action of $\Gamma_0(N)$.

 For $g\in \SL_2(\mathbb{Z})$, and $w$ the width of the cusp $p$ determined by $g$, a coordinate chart of the `compactification' $\mathcal{X}_0(N)_{\mathbb{C}}^{\an}$ of $\mathcal{Y}_0(N)_{\mathbb{C}}^{\an}$ at a neighbourhood of $p$ is induced by
 \begin{align}\label{eq:coordinate-cusp}
\tau \mapsto e^{2\pi i\frac{g^{-1}\cdot \tau}{w}}\text{.}
\end{align}
 
\begin{obs} \label{obs:level-struct}
  The moduli stacks $\mathcal{Y}_0(N)$ and $\mathcal{X}_0(N)$ are \emph{never} representable by schemes. Indeed, any subgroup of an elliptic curve $E$ is preserved by multiplication by $-1$; thus, any object of $\mathcal{Y}_0(N)$ admits non-trivial automorphisms. If $p$ is an odd prime not dividing $N$, the moduli stack $\mathcal{Y}_0(N)_p$ classifying elliptic curves with a cyclic subgroup of order $N$ \emph{and} a full level $p$ structure is representable by a smooth affine $\mathbb{Q}$-scheme $Y$ with a $\GL_2(\mathbb{F}_p)$-action (cf. \cite{DR73} Théorème 2.7). The natural map $Y\to \mathcal{Y}_0(N)$ is finite étale and induces an isomorphism $\mathcal{Y}_0(N) \cong Y/\!/\GL_2(\mathbb{F}_p)$. 
  Similar considerations also hold for the compactified modular curve $\mathcal{X}_0(N)$.
\end{obs}

Let $\mathcal{F}$ be the \emph{Hodge line bundle} over $\mathcal{Y}_0(N)$. By definition, if $\pi: \mathcal{E} \to \mathcal{Y}_0(N)$ denotes the universal elliptic curve over $\mathcal{Y}_0(N)$, then $\mathcal{F} \defeq \pi_* \Omega^{1}_{\mathcal{E}/\mathcal{Y}_0(N)}$. Concretely, for a $k$-point $y:\Spec k \to \mathcal{Y}_0(N)$ corresponding to a pair $(E,C)$ over a field $k$, the fibre of $\mathcal{F}$ at $y$ is given by $\mathcal{F}(y) = H^0(E,\Omega^1_{E/k})$. The Hodge line bundle can be extended to a line bundle $\overline{\mathcal{F}}$ over $\mathcal{X}_0(N)$ by setting $\overline{\mathcal{F}}\defeq \overline{\pi}_* \Omega^{1}_{\overline{\mathcal{E}}/\mathcal{X}_0(N)}(\log \overline{\pi}^{-1}(\mathcal{Z}_N))$, where $\overline{\pi}: \overline{\mathcal{E}} \to \mathcal{X}_0(N)$ denotes the universal generalised elliptic curve over $\mathcal{X}_0(N)$.

\begin{ex}\label{eq:action-gamma-omega}
  The pullback of $\mathcal{F}^{\an}$ to $\mathbb{H}$ via the uniformisation map (\ref{eq:unif-Y_0(N)}) is trivialised by the global section $\omega \in H^0(\mathbb{H}, \mathcal{F}^{\an})$ whose fibre at $\tau \in \mathbb{H}$ is the holomorphic 1-form
  $$
\omega_{\tau} \defeq 2\pi i\, dz
$$
on the complex torus $\mathbb{E}_{\tau} = \mathbb{C}/(\mathbb{Z} + \tau \mathbb{Z})$. Note that every $\gamma \in \Gamma_0(N)$ defines an automorphism of $\mathbb{E}$ by
$$
\gamma_\tau : \mathbb{E}_{\tau} \to \mathbb{E}_{\gamma\cdot \tau}\text{, }\qquad z \mapsto j(\gamma,\tau)^{-1}z\text{,}
$$
so that it acts on $\mathcal{F}^{\an}$ by pullback. Explicitly, $\gamma^*\omega = j(\gamma,\tau)^{-1}\omega$.
\end{ex}

Let $r$ be an integer. If $f \in M^!_{r}(\Gamma_0(N))$, then it follows from the above example that the section $f\omega^{\tensor r} \in H^0(\mathbb{H}, (\mathcal{F}^{\an})^{\tensor r})$ is $\Gamma_0(N)$-invariant, so that it descends to a global section $\omega_f$ in $ H^0(\mathcal{Y}_0(N)^{\an}_{\mathbb{C}}, (\mathcal{F}^{\an})^{\tensor r})$.

\begin{prop}[$q$-expansion principle; cf. \cite{katz73} 1.6]\label{prop:q-expansion-principle}
  The section $\omega_f$ constructed above is algebraic, i.e., $\omega_f \in H^0(\mathcal{Y}_0(N)_{\mathbb{C}},\mathcal{F}^{\tensor r}) = H^0(\mathcal{Y}_0(N),\mathcal{F}^{\tensor r})\tensor_{\mathbb{Q}}\mathbb{C}$. Moreover, $f \in M^!_r(\Gamma_0(N);\mathbb{Q})$ if and only if $\omega_f \in H^0(\mathcal{Y}_0(N), \mathcal{F}^{\tensor r})$.
\end{prop}

\begin{proof}[Sketch of the proof]
    Let $p$ be an odd prime not dividing $N$, so that $\mathcal{X}_0(N)_p$ (resp. $\mathcal{Y}_0(N)_p$) is representable by a $\mathbb{Q}$-scheme $X$ (resp. $Y$) (see Remark \ref{obs:level-struct}).  By the growth condition of $f$ at the cusps, $\omega_f$ extends to a \emph{meromorphic} global section of $(\overline{\mathcal{F}}^{\an})^{\tensor r}$ over $X_{\mathbb{C}}^{\an}$. Since $\overline{\mathcal{F}}$ is an algebraic coherent sheaf, and since $X_{\mathbb{C}}$ is projective, Serre's GAGA implies that the pullback of $\omega_f$ to $X_{\mathbb{C}}^{\an}$ is given by a \emph{rational} global section of $\overline{\mathcal{F}}^{\tensor r}$ over $X_{\mathbb{C}}$; since it is regular over $Y_{\mathbb{C}}$, it belongs to $H^0(Y_{\mathbb{C}}, \mathcal{F}^{\tensor r})$. By considering $\GL_{2}(\mathbb{F}_p)$-invariants, we conclude that $\omega_f \in H^0(\mathcal{Y}_0(N)_{\mathbb{C}},\mathcal{F}^{\tensor r})$.

  Note that the uniformisation map $\mathbb{H} \to \mathcal{Y}_0(N)_{\mathbb{C}}^{\an}$ (\ref{eq:unif-Y_0(N)}) factors through a map $\mathbb{D}^* \to  \mathcal{Y}_0(N)_{\mathbb{C}}^{\an}$ via $\mathbb{H} \to \mathbb{D}^*$, $\tau \mapsto q(\tau)=e^{2\pi i \tau}$, where $\mathbb{D}^*$ denotes the punctured complex unit disc. The section $\omega$ over $\mathbb{H}$ also descends to $\mathbb{D}^*$, so that the pullback of $\omega_f$ to $\mathbb{D}^*$ is simply $f(q)\omega^{\tensor r}$, where $f(q)$ denotes the Fourier expansion of $f$ at infinity. We now remark that $\mathbb{D}^* \to  \mathcal{Y}_0(N)_{\mathbb{C}}^{\an}$ `descends' to an algebraic morphism $\Spec \mathbb{Q}(\!(q)\!) \to \mathcal{Y}_0(N)$, given by the \emph{Tate curve} (see \cite{DR73} VII.4). From this, we conclude that $\omega_f \in H^0(\mathcal{Y}_0(N), \mathcal{F}^{\tensor r})$ if and only if $f(q) \in \mathbb{Q}(\!(q)\!)$. 
\end{proof}

\begin{coro}\label{coro:field-definition}
  For every integer $r$ and every subfield $K$ of $\mathbb{C}$, the map
  $$
M_{r}^!(\Gamma_0(N);K) \to H^0(\mathcal{Y}_0(N)_K,\mathcal{F}^{\tensor r})\text{, }\qquad f \mapsto \omega_f 
$$
is an isomorphism of $K$-vector spaces. In particular, $M_r^!(\Gamma_0(N);K) = M_r^!(\Gamma_0(N);\mathbb{Q})\tensor_{\mathbb{Q}} K$. \hfill $\square$
\end{coro}

It follows from the explicit description of the coordinates at the cusps of $\mathcal{X}_0(N)_{\mathbb{C}}^{\an}$ given in (\ref{eq:coordinate-cusp}) that, if $f \in M_r(\Gamma_0(N))$, then $\omega_f$ extends holomorphically to every cusp. Thus,
$$
M_r(\Gamma_0(N);K) \cong H^0(\mathcal{X}_0(N)_K,\overline{\mathcal{F}}^{\tensor r})\text{,}
$$
for every subfield $K\subset \mathbb{C}$.  Similarly, $M_r^{!,\infty}(\Gamma_0(N))$ corresponds to those elements of $H^0(\mathcal{Y}_0(N)_{\mathbb{C}},\mathcal{F}^{\tensor r})$ which extend holomorphically at every cusp different from $\infty$, and the space $S_r^{!}(\Gamma_0(N))$ is given by the elements of $H^0(\mathcal{Y}_0(N)_{\mathbb{C}},\mathcal{F}^{\tensor r})$ with trivial residue at every cusp.

The following technical lemma will be helpful throughout this paper.

\begin{lemma}\label{lemma:principal-part}
Let $k\ge 0$ be an even integer, and $f \in M_{-k}^{!,\infty}(\Gamma_0(N))$. If $K$ denotes the subfield of $\mathbb{C}$ generated by $a_n(f)$ for $n\le 0$, then $f \in M_{-k}^{!,\infty}(\Gamma_0(N);K)$.
\end{lemma}

In other words, for any weakly holomorphic modular form of negative or zero weight, its field of definition is determined by its principal part.

\begin{proof}
  Since there are no nonzero modular forms of negative weight, and the only modular forms in weight 0 are constant (Remark \ref{obs:negative-weight})\footnote{This also follows the above geometric description of $M_r(\Gamma_0(N);K)$, since the Hodge line bundle $\overline{\mathcal{F}}$ is ample over a finite étale cover $X$ of $\mathcal{X}_0(N)$ (\cite{DR73} VII.3.4).}, the $\mathbb{C}$-linear map $M_{-k}^{!,\infty}(\Gamma_0(N)) \to \mathbb{C}[q^{-1}]$, $f\mapsto \mathcal{P}_f$, given by the principal part at the cusp at infinity is injective.  It follows from Corollary \ref{coro:field-definition} that, for any subfield $K$ of $\mathbb{C}$, this injection is the $\mathbb{C}$-linear extension of the analogous $K$-linear injection  $M_{-k}^{!,\infty}(\Gamma_0(N);K) \to K[q^{-1}]$. Thus, if $f \in M_{-k}^{!,\infty}(\Gamma_0(N))$ is such that $\mathcal{P}_f \in K[q^{-1}]\subset \mathbb{C}[q^{-1}]$, then $f$ itself belongs to $M_{-k}^{!,\infty}(\Gamma_0(N);K) \subset M_{-k}^{!,\infty}(\Gamma_0(N))$.
\end{proof}

\section{De Rham cohomology with coefficients on modular curves} \label{sec:derham}

Let $(\mathcal{V},\nabla)$ be the \emph{de Rham bundle} over $\mathcal{Y}_0(N)$, with its Gauss--Manin connection. Formally, $\mathcal{V} \defeq \mathbb{R}^1\pi_* \Omega^{\bullet}_{\mathcal{E}/\mathcal{Y}_0(N)}$, where $\pi:\mathcal{E} \to \mathcal{Y}_0(N)$ denotes the universal elliptic curve over $\mathcal{Y}_0(N)$. If $y: \Spec k \to \mathcal{Y}_0(N)$ is a $k$-point corresponding to a pair $(E,C)$ over $k$, then the fibre of $\mathcal{V}$ at $y$ is given by the algebraic de Rham cohomology $\mathcal{V}(y) = H^1_{\dR}(E/k)$.

One can show (cf. \cite{katz73} A1.2 - A1.3) that $\mathcal{V}$ is a rank 2 vector bundle over $\mathcal{Y}_0(N)$ containing the Hodge line bundle $\mathcal{F}$ as a subbundle, and equipped with a symplectic $\mathcal{O}_{\mathcal{Y}_0(N)}$-bilinear pairing
$$
\langle \ , \ \rangle : \mathcal{V} \times\mathcal{V} \to \mathcal{O}_{\mathcal{Y}_0(N)}
$$
induced by the canonical principal polarisation of $\mathcal{E}$, satisfying $d\langle v,w\rangle = \langle \nabla v, w \rangle + \langle v, \nabla w \rangle$. For every integer $k\ge 0$, we set
$$
\mathcal{V}_k \defeq {\Sym}^k\mathcal{V}\text{,}
$$
and we denote by
$$
\langle \ , \ \rangle_k: \mathcal{V}_k\times \mathcal{V}_k \to \mathcal{O}_{\mathcal{Y}_0(N)}\text{, }\qquad (\alpha_1\cdots \alpha_k, \beta_1 \cdots \beta_k)\mapsto \frac{1}{k!}\sum_{\sigma \in \mathfrak{S}_k}\prod_{j=1}^k\langle \alpha_j,\beta_{\sigma(j)}\rangle
$$
the induced symplectic pairing on $\mathcal{V}_k$. By abuse, we continue to denote by $\nabla$ the induced connection on $\mathcal{V}_k$.

Consider the $\mathbb{Q}$-vector space
$$
H^1_{\dR}(\mathcal{Y}_0(N),\mathcal{V}_k) \defeq \mathbb{H}^1(\mathcal{Y}_0(N),\mathcal{V}_k\tensor \Omega^{\bullet}_{\mathcal{Y}_0(N)/\mathbb{Q}})\text{,}
$$
where $\mathcal{V}_k\tensor \Omega^{\bullet}_{\mathcal{Y}_0(N)/\mathbb{Q}}$ is the complex of $\mathcal{O}_{\mathcal{Y}_0(N)}$-modules concentrated in degrees 0 and 1
$$
[\mathcal{V}_k \stackrel{\nabla}{\to} \mathcal{V}_k \tensor \Omega^1_{\mathcal{Y}_0(N)/\mathbb{Q}}]\text{.}
$$
Since $\mathcal{Y}_0(N)$ admits an étale cover by a smooth \emph{affine} scheme (Remark \ref{obs:level-struct}), we can identify
$$
H^1_{\dR}(\mathcal{Y}_0(N),\mathcal{V}_k) = \text{coker} (H^0(\mathcal{Y}_0(N),\mathcal{V}_k) \stackrel{\nabla}{\to}H^0(\mathcal{Y}_0(N),\mathcal{V}_k\tensor \Omega^1_{\mathcal{Y}_0(N)/\mathbb{Q}}))\text{;}
$$
in particular, we have a natural surjection
\begin{align}\label{eq:cohom-quotient}
H^0(\mathcal{Y}_0(N),\mathcal{V}_k\tensor \Omega^1_{\mathcal{Y}_0(N)/\mathbb{Q}}) \longtwoheadrightarrow H^1_{\dR}(\mathcal{Y}_0(N),\mathcal{V}_k)\text{.}
\end{align}

Recall that there is a \emph{Kodaira--Spencer isomorphism}
$$
\mathcal{F}^{\tensor 2} \stackrel{\sim}{\to} \Omega^1_{\mathcal{Y}_0(N)/\mathbb{Q}}\text{, } \qquad \omega_1\tensor \omega_2 \mapsto \langle \omega_1,\nabla \omega_2 \rangle\text{,}
$$
whose pullback to $\mathbb{H}$ identifies $\omega^{\tensor 2}$ with $2\pi i\, d\tau = d \log q$ (\cite{katz73} A1.3.17). Thus,
\begin{align}\label{eq:identification-ks}
M_{k+2}^!(\Gamma_0(N);\mathbb{Q}) \cong H^0(\mathcal{Y}_0(N),\mathcal{F}^{\tensor k +2}) \cong H^0(\mathcal{Y}_0(N),\mathcal{F}^{\tensor k }\tensor \Omega^1_{\mathcal{Y}_0(N)/\mathbb{Q}})\text{.}
\end{align}
Note that the pullback of $\omega_f$ to the Poincaré half-plane $\mathbb{H}$  via (\ref{eq:unif-Y_0(N)}) gets identified with $f \omega^{\tensor k} \tensor 2\pi i\, d\tau$. Given the natural inclusion of $\mathcal{F}^{\tensor k}$ into $\mathcal{V}_k$, we obtain from (\ref{eq:cohom-quotient}) and (\ref{eq:identification-ks}) a map
$$
M_{k+2}^!(\Gamma_0(N);\mathbb{Q}) \to H^1_{\dR}(\mathcal{Y}_0(N),\mathcal{V}_k)\text{, }\qquad f\mapsto [\omega_f]=:[f]\text{.}
$$

Recall that we have the holomorphic derivation $D = \frac{1}{2\pi i }\frac{d}{d\tau}$ on $\mathbb{H}$. A classical identity of Bol implies that $D^{k+1}f \in M^!_{k+2}(\Gamma_0(N))$ for every $f \in M_{-k}^!(\Gamma_0(N))$ (cf. \cite{BH18} Proposition 2.2).

\begin{theorem}[cf. \cite{scholl85}, \cite{coleman96}, \cite{candelori14}, \cite{KS16}, \cite{BH18}]\label{thm:derham-cohom}
  The sequence of $\mathbb{Q}$-vector spaces
  \begin{align*}
    0 \to M_{-k}^!(\Gamma_0(N);\mathbb{Q}) \stackrel{D^{k+1}}{\to} M^!_{k+2}(\Gamma_0(N);\mathbb{Q}) &\to H^1_{\dR}(\mathcal{Y}_0(N),\mathcal{V}_k) \to 0\\
    f & \mapsto [f]
  \end{align*}
  is exact.
\end{theorem}

Note that $M_{k+2}(\Gamma_0(N);\mathbb{Q})$ maps injectively into $H^1_{\dR}(\mathcal{Y}_0(N),\mathcal{V}_k)$.

The above theorem appears to have been rediscovered by several authors. For later reference, we recall the main steps its proof.

\begin{proof}[Sketch of proof]
  We work over $Y$, an affine $\mathbb{Q}$-scheme representing $\mathcal{Y}_0(N)_p$ for some odd prime $p$ not dividing $N$ (Remark \ref{obs:level-struct}); our statement is then obtained by taking $\GL_2(\mathbb{F}_p)$-invariants.

  Consider the Hodge filtration  $F^i\mathcal{V}_k \defeq \im (\mathcal{F}^{\tensor i}\tensor \mathcal{V}_{k-i} \to \mathcal{V}_k)$ on $\mathcal{V}_k$. The key fact is that the Gauss--Manin connection $\nabla$ induces an isomorphism of $\mathcal{O}_Y$-modules
\begin{align}\label{eq:key-fact}
F^1\mathcal{V}_k = F^1\mathcal{V}_k/F^{k+1}\mathcal{V}_k \stackrel{\sim}{\to} \mathcal{V}_k/F^k\mathcal{V}_k  \tensor \Omega^1_{Y/\mathbb{Q}}\text{.}
\end{align}
This is a consequence of the Kodaira--Spencer isomorphism, for it implies that $\nabla$ induces the successive isomorphisms $ F^i\mathcal{V}_k/F^{i+1}\mathcal{V}_k \stackrel{\sim}{\to} (F^{i-1}\mathcal{V}_k/ F^{i}\mathcal{V}_k)\tensor \Omega^1_{Y/\mathbb{Q}}$ (cf. \cite{coleman96} Lemma 4.2).

The surjectivity of the natural map
\begin{align}\label{eq:surjectiveness}
H^0(Y,\mathcal{F}^{\tensor k}\tensor \Omega^1_{Y/\mathbb{Q}}) \to H^1_{\dR}(Y,\mathcal{V}_k)
\end{align}
follows immediately from (\ref{eq:key-fact}): every $\alpha \in H^0(Y,\mathcal{V}_k\tensor \Omega^1_{Y/\mathbb{Q}})$ can be written as $\alpha = \nabla Q + \beta$ for some $Q \in H^0(Y,F^1\mathcal{V}_k)$ and $\beta \in H^0(Y,F^k\mathcal{V}_k\tensor \Omega^1_{Y/\mathbb{Q}}) = H^0(Y,\mathcal{F}^k\tensor \Omega^1_{Y/\mathbb{Q}})$, so that $\alpha$ and $\beta$ map to the same class in $H^1_{\dR}(Y,\mathcal{V}_k)$. Its kernel is given by the image of the $\mathcal{O}_Y$-morphism
\begin{align*}
 \delta:\mathcal{V}_k/F^1\mathcal{V}_k \to F^k\mathcal{V}_k\tensor \Omega^1_{Y/\mathbb{Q}} = \mathcal{F}^{\tensor k}\tensor \Omega^1_{Y/\mathbb{Q}}\text{, }\qquad [P]\mapsto \nabla P - \nabla Q
\end{align*}
where $P$ denotes a section of $\mathcal{V}_k$, $[P]$ its class in $\mathcal{V}_k/F^1\mathcal{V}_k$, and $Q$ the unique section of $F^1\mathcal{V}_k$ such that $\nabla P - \nabla Q$ lies in $F^k\mathcal{V}_k\tensor \Omega^1_{Y/\mathbb{Q}}$; the existence and uniqueness of $Q$ is guaranteed by the isomorphism (\ref{eq:key-fact}).

The map $\delta$ is essentially the Bol operator $D^{k+1}$. Indeed, we can identify
$$
\mathcal{V}_k/F^1\mathcal{V}_k \stackrel{\sim}{\to} (F^k\mathcal{V}_k)^{\vee}\text{, } \qquad [P]\mapsto \langle  \ , P \rangle_k\text{;}
$$
pulling back to $\mathbb{H}$, an explicit computation shows that if $f$ is a weakly holomorphic modular form of weight $-k$, then  $\delta(f\omega^{\tensor -k}) =  \frac{(-1)^k}{k!}D^{k+1}f \omega^{\tensor k}\tensor 2\pi i \, d\tau$ (cf. \cite{coleman96} Proposition 4.3).
\end{proof}

We now consider cuspidal cohomology. Let $\overline{\mathcal{V}} \defeq \mathbb{R}^1 \overline{\pi}_*\Omega^{\bullet}_{\overline{\mathcal{E}}/\mathcal{X}_0(N)}(\log \overline{\pi}^{-1}(\mathcal{Z}_N))$ with its logarithmic Gauss--Manin connection $\overline{\nabla}: \overline{\mathcal{V}} \to \overline{\mathcal{V}}\tensor \Omega^1_{\mathcal{X}_0(N)/\mathbb{Q}}(\log \mathcal{Z}_N)$. Then, $(\overline{\mathcal{V}}, \overline{\nabla})$ is a rank 2 vector bundle with logarithmic connection over $\mathcal{X}_0(N)$ extending $(\mathcal{V},\nabla)$. In fact, one can prove that the residues of $\overline{\nabla}$ at the cusps are nilpotent (cf. \cite{katz73} A1.4), so that $(\overline{\mathcal{V}},\overline{\nabla})$ is Deligne's canonical extension of $(\mathcal{V},\nabla)$; in particular, 
$$
H^1_{\dR}(\mathcal{Y}_0(N),\mathcal{V}_k) = \mathbb{H}^1(\mathcal{X}_0(N), \overline{\mathcal{V}}_k\tensor \Omega^{\bullet}_{\mathcal{X}_0(N)/\mathbb{Q}}(\log \mathcal{Z}_N) )\text{,}
$$
where $\overline{\mathcal{V}}_k \defeq \Sym^k \overline{\mathcal{V}}$. The residue $\overline{\mathcal{V}}_k\tensor \Omega^{1}_{\mathcal{X}_0(N)/\mathbb{Q}}(\log \mathcal{Z}_N) \to \overline{\mathcal{V}}_k|_{\mathcal{Z}_N}$ induces a surjective $\mathbb{Q}$-linear map
$$
\Res : H^1_{\dR}(\mathcal{Y}_0(N),\mathcal{V}_k) \to H^0(\mathcal{Z}_N,\overline{\mathcal{V}}_k|_{\mathcal{Z}_N})/\im(\Res \overline{\nabla})\cong H^0(\mathcal{Z}_N,\overline{\mathcal{F}}^{\otimes k}|_{\mathcal{Z}_N})\text{,}
$$
and we define
$$
H^1_{\dR,\cusp}(\mathcal{Y}_0(N),\mathcal{V}_k) \defeq \ker ( H^1_{\dR}(\mathcal{Y}_0(N),\mathcal{V}_k) \stackrel{\Res}{\to} H^0(\mathcal{Z}_N,\overline{\mathcal{F}}^{\otimes k}|_{\mathcal{Z}_N}))\text{.}
$$
See \cite{scholl85} 2.6, or \cite{KS16} 4, for a direct definition of the cuspidal cohomology as the hypercohomology of a subcomplex of $[\overline{\mathcal{V}}_k \stackrel{\overline{\nabla}}{\to} \overline{\mathcal{V}}_k \otimes \Omega^{1}_{\mathcal{X}_0(N)/\mathbb{Q}}(\log \mathcal{Z}_N)]$.

Let $f \in M_{k+2}^!(\Gamma_0(N);\mathbb{Q})$. Using the explicit description of the coordinates at the cusps (\ref{eq:coordinate-cusp}), we can identify
$$
\Res([f]) =(a_{0,g_1}(f),\ldots,a_{0,g_m}(f))\text{,}
$$
where $\{g_1,\ldots,g_m\}$ is a set of representatives of $\Gamma_0(N)\backslash\SL_2(\mathbb{Z})/\Gamma_{\infty}$. In particular, the subspace $S_{k+2}^!(\Gamma_0(N);\mathbb{Q})\subset M_{k+2}^!(\Gamma_0(N);\mathbb{Q})$ maps to the subspace $H^1_{\dR,\cusp}(\mathcal{Y}_0(N),\mathcal{V}_k)\subset H^1_{\dR}(\mathcal{Y}_0(N),\mathcal{V}_k)$. It follows from this observation, and from Theorem \ref{thm:derham-cohom} that the sequence of $\mathbb{Q}$-vector spaces
\begin{align}\label{eq:exact-seq-cusp}
  0 \to M_{-k}^!(\Gamma_0(N);\mathbb{Q}) \stackrel{D^{k+1}}{\to} S^!_{k+2}(\Gamma_0(N);\mathbb{Q}) & \to H^1_{\dR,\cusp}(\mathcal{Y}_0(N),\mathcal{V}_k) \to 0
\end{align}
is exact.

\begin{obs}
The spectral sequence associated to the Hodge filtration induces an exact sequence (cf. proof of Theorem 2.7 in \cite{scholl85})
$$
0 \to S_{k+2}(\Gamma_0(N);\mathbb{Q}) \to H^1_{\dR,\cusp}(\mathcal{Y}_0(N),\mathcal{V}_k) \to S_{k+2}(\Gamma_0(N);\mathbb{Q})^{\vee} \to 0\text{.}
$$
In particular, 
$$
\dim H^1_{\dR,\cusp}(\mathcal{Y}_0(N),\mathcal{V}_k) = 2 \dim S_{k+2}(\Gamma_0(N);\mathbb{Q})\text{.}
$$
\end{obs}

Theorem \ref{thm:derham-cohom} (and the exact sequence (\ref{eq:exact-seq-cusp})) can be refined: we only need to consider weakly holomorphic cusp forms whose poles are concentrated at the cusp $\infty$. This follows from the next lemma.

\begin{lemma}\label{lemma:pole-concentration}
  We have
  $$
M_{k+2}^!(\Gamma_0(N);\mathbb{Q}) \subset M_{k+2}^{!,\infty}(\Gamma_0(N);\mathbb{Q}) + D^{k+1}M^!_{-k}(\Gamma_0(N);\mathbb{Q})\text{.}
  $$
\end{lemma}

\begin{proof}
  Let $p$ be an odd prime not dividing $N$, and $X$ (resp. $Y$) be a $\mathbb{Q}$-scheme representing the étale cover $\mathcal{X}_0(N)_p$ (resp. $\mathcal{Y}_0(N)_p$) of $\mathcal{X}_0(N)$ (resp. $\mathcal{Y}_0(N)$. Recall that $Y$ is affine, and $X$ is the projective compactification of $Y$ (Remark \ref{obs:level-struct}).

  We denote by $Z$ the cuspidal locus, so that $X\setminus Z = Y$.  Consider the affine curve $X^*=X\setminus\{\infty\}$ containing $Y$. We can identify $H^0(X^*, \overline{\mathcal{F}}^{\tensor k} \tensor \Omega^1_{X/\mathbb{Q}}(\log Z))$ as a $\mathbb{Q}$-subspace of $H^0(Y, \mathcal{F}^{\tensor k} \tensor \Omega^1_{Y/\mathbb{Q}})$. By taking $\GL_2(\mathbb{F}_p)$-invariants, our statement will follow from the surjectivity of the natural map
$$
H^0(X^*, \overline{\mathcal{F}}^{\tensor k}\tensor \Omega^1_{X/\mathbb{Q}}(\log Z)) \to H^1_{\dR}(Y,\mathcal{V}_k)\text{.}
$$

As $(\overline{\mathcal{V}}_k,\overline{\nabla})|_{X^*}$ is Deligne's canonical extension of $(\mathcal{V}_k,\nabla)$ to $X^*$,  we have a natural isomorphism
$$
\mathbb{H}^1(X^*,\overline{\mathcal{V}}_k\tensor \Omega^{\bullet}_{X/\mathbb{Q}}(\log Z))\stackrel{\sim}{\to} H_{\dR}^1(Y,\mathcal{V}_k)\text{.}
$$
Since $X^*$ is affine,  $\mathbb{H}^1(X^*,\overline{\mathcal{V}}_k\tensor \Omega^{\bullet}_{X/\mathbb{Q}}(\log Z))$ gets identified with the cokernel of the $k$-linear map $\overline{\nabla}: H^0(X^*,\overline{\mathcal{V}}_k) \to H^0(X^*,\overline{\mathcal{V}}_k\tensor \Omega^1_{X/\mathbb{Q}}(\log Z))$; in particular, we obtain a surjective map
$$
H^0(X^*,\overline{\mathcal{V}}_k\tensor \Omega^1_{X/\mathbb{Q}}(\log Z)) \to \mathbb{H}^1(X^*,\overline{\mathcal{V}}_k\tensor \Omega^{\bullet}_{X/\mathbb{Q}}(\log Z))\text{.}
$$
The surjectivity of its restriction to the subspace $H^0(X^*, \overline{\mathcal{F}}^{\tensor k}\tensor \Omega^1_{X/\mathbb{Q}}(\log Z))$
$$
H^0(X^*, \overline{\mathcal{F}}^{\tensor k}\tensor \Omega^1_{X/\mathbb{Q}}(\log Z)) \to \mathbb{H}^1(X^*,\overline{\mathcal{V}}_k\tensor \Omega^{\bullet}_{X/\mathbb{Q}}(\log Z))
$$
is then proved in the same way as the surjectivity of (\ref{eq:surjectiveness}).
\end{proof}

\begin{coro}\label{coro:derahm-whmf}
  The map $M^!_{k+2}(\Gamma_0(N);\mathbb{Q}) \to H^1_{\dR}(\mathcal{Y}_0(N),\mathcal{V}_k)$, $f\mapsto [f]$, induces short exact sequences of $\mathbb{Q}$-vector spaces
  $$
0 \to M_{-k}^{!,\infty}(\Gamma_0(N);\mathbb{Q}) \stackrel{D^{k+1}}{\to} M^{!,\infty}_{k+2}(\Gamma_0(N);\mathbb{Q})  \to H^1_{\dR}(\mathcal{Y}_0(N),\mathcal{V}_k) \to 0
$$
and
$$
0 \to M_{-k}^{!,\infty}(\Gamma_0(N);\mathbb{Q}) \stackrel{D^{k+1}}{\to} S^{!,\infty}_{k+2}(\Gamma_0(N);\mathbb{Q}) \to H^1_{\dR,\cusp}(\mathcal{Y}_0(N),\mathcal{V}_k) \to 0\text{.}
$$
\hfill $\square$
\end{coro}

We finish this section with a brief discussion of the de Rham pairing
\begin{align}\label{eq:de-rham-pairing}
\langle \ , \ \rangle_{\dR} : H^1_{\dR,\cusp}(\mathcal{Y}_0(N),\mathcal{V}_k)\times H^1_{\dR}(\mathcal{Y}_0(N),\mathcal{V}_k) \to \mathbb{Q}\text{,}
\end{align}
whose restriction to $H^1_{\dR,\cusp}(\mathcal{Y}_0(N),\mathcal{V}_k)\times H^1_{\dR,\cusp}(\mathcal{Y}_0(N),\mathcal{V}_k)$ is a symplectic $\mathbb{Q}$-bilinear form.

Explicitly, we can compute $\langle \ , \ \rangle_{\dR}$ (defined in terms of the cup-product in de Rham cohomology) via the usual recipe involving residues at the cusps (cf. \cite{coleman96} Theorem 5.2). In fact, by Lemma \ref{lemma:pole-concentration}, we only need to consider the cusp $\infty$. Let $\varphi: \Spec \mathbb{Q}(\!(q)\!) \to \mathcal{Y}_0(N)$ be the morphism given by the Tate curve (as in the proof of Proposition \ref{prop:q-expansion-principle}). For $f \in S_{k+2}^{!,\infty}(\Gamma_0(N);\mathbb{Q})$ and $g \in M_{k+2}^{!,\infty}(\Gamma_0(N);\mathbb{Q})$, we have
$$
\langle [f],[g]\rangle_{\dR} = {\Res}_{q=0}\langle F,\varphi^*\omega_g \rangle_k\text{,}
$$
where  $F \in H^0(\Spec \mathbb{Q}(\!(q)\!), \varphi^*\mathcal{V}_k)$ satisfies
$$
\nabla F = \varphi^*\omega_f = f(q)\omega^{k}\tensor d \log q\text{.}
$$
Solving for $F$, we obtain the explicit formula
\begin{align}\label{eq:derham-cup}
\langle [f],[g]\rangle_{\dR} = k!\sum_{n \in \mathbb{Z}}\frac{a_n(f)a_{-n}(g)}{n^{k+1}}\text{.}
\end{align}
Note that, since $g$ is meromorphic at infinity, the above sum is finite.

\section{Realisations of modular motives}\label{sec:real-modular-motives}

Next, we construct an object $H^1(\mathcal{Y}_0(N),V_k)$ of $\mathcal{H}(\mathbb{Q})$ whose de Rham realisation is
$$
H^1(\mathcal{Y}_0(N),V_k)_{\dR} = H^1_{\dR}(\mathcal{Y}_0(N),\mathcal{V}_k)\text{,}
$$
as defined in the last section. Throughout, $k\ge 0$ is an \emph{even} integer.

For the Betti realisation, let $\pi: \mathcal{E} \to \mathcal{Y}_0(N)$ be the universal elliptic curve over $\mathcal{Y}_0(N)$, and consider the $\mathbb{Q}$-local system
$$
\mathbb{V} \defeq R^1\pi^{\an}_{\mathbb{C},*} \mathbb{Q}_{\mathcal{E}^{\an}_{\mathbb{C}}}
$$
over $\mathcal{Y}_0(N)_{\mathbb{C}}^{\an}$, whose stalk at a point $y$ of $\mathcal{Y}_0(N)_{\mathbb{C}}^{\an}$, corresponding to a pair $(E,C)$ defined over $\mathbb{C}$, is the Betti cohomology with rational coefficients $\mathbb{V}_y = H^1(E^{\an},\mathbb{Q})$. Setting $\mathbb{V}_k \defeq \Sym^k \mathbb{V}$, we define
$$
H^1(\mathcal{Y}_0(N),V_k)_{\B} = H^1(\mathcal{Y}_0(N)_{\mathbb{C}}^{\an},\mathbb{V}_k)\text{.}
$$

The Betti realisation has a concrete description in terms of group cohomology. The pullback of $\mathbb{V}$ to $\mathbb{H}$ via the uniformisation map $\mathbb{H} \to \mathcal{Y}_0(N)_{\mathbb{C}}^{\an}$ (\ref{eq:unif-Y_0(N)}) is trivialised by the global sections
$$
X\defeq\sigma_1^{\vee}\ \ \ \text{  and }\ \ \ Y \defeq -\sigma_{\tau}^{\vee}\text{,}
$$
where $\sigma_1,\sigma_{\tau} \in H^0(\mathbb{H},\mathbb{V}^{\vee})$ are the constant families of 1-cycles given by $1$ and $\tau$ under the identification $H^1(\mathbb{E}_{\tau},\mathbb{Q})^{\vee}=\mathbb{Q} + \tau \mathbb{Q}$. Thus, $\mathbb{V}$ can be regarded as the vector space $V = \mathbb{Q}X \oplus \mathbb{Q} Y$ endowed with the right $\Gamma_0(N)$-action
$$
(X,Y)|_{\gamma} = (aX + bY,cX + dY)\text{, } \qquad \gamma = \left(\begin{array}{cc} a & b \\ c & d\end{array} \right) \in \Gamma_0(N)\text{.}
$$
In particular, if we set $V_k \defeq \Sym^k V_k$, then we have a canonical identification
\begin{align}\label{eq:group-cohom}
H^1(\mathcal{Y}_0(N)_{\mathbb{C}}^{\an},\mathbb{V}_k) \cong H^1(\Gamma_0(N),V_k)\text{.}
\end{align}

Fixing an arbitrary $\tau_0\in \mathbb{H}$, the comparison isomorphism
\begin{align}\label{eq:comparison}
\comp : H^1_{\dR}(\mathcal{Y}_0(N),\mathcal{V}_k) \tensor_{\mathbb{Q}} \mathbb{C} \stackrel{\sim}{\to} H^1(\mathcal{Y}_0(N)_{\mathbb{C}}^{\an},\mathbb{V}_k)\tensor_{\mathbb{Q}}\mathbb{C}
\end{align}
can be explicitly given by
$$
\comp([f]) = [C_f]\text{,}
$$
where $[f]$ denotes the class of some $f \in M_{k+2}^!(\Gamma_0(N))$ (see Sect. \ref{sec:derham}), and $C_f : \Gamma_0(N) \to V_k$ is the 1-cocycle defined by
\begin{align}\label{eq:1-cocyle}
C_{f}:\Gamma_0(N) \to V_k\text{, }\qquad \gamma \mapsto (2\pi i)^{k+1}\int_{\gamma^{-1}\cdot\tau_0}^{\tau_0}f(\tau)(X - \tau Y)^kd\tau\text{.}
\end{align}
This follows from the fact that the `relative comparison isomorphism'
\begin{align}\label{eq:rel-comp}
  \comp : (\mathcal{V}_{\mathbb{C}}^{\an},\nabla^{\an}) \stackrel{\sim}{\to} (\mathbb{V}\tensor_{\mathbb{Q}} \mathcal{O}_{\mathcal{Y}_0(N)^{\an}_{\mathbb{C}}},\id \tensor d)
\end{align}
satisfies $\comp(\omega) = \left(\int_{\sigma_1}\omega\right)\sigma_1^{\vee} + \left(\int_{\sigma_{\tau}}\omega \right)\sigma_\tau^{\vee} = 2\pi i (X - \tau Y)$.

The weight and Hodge filtrations on $H^1_{\dR}(\mathcal{Y}_0(N),\mathcal{V}_k)$ are explicitly defined by
\begin{align*}
&W^{\dR}_{k} = 0  \\
&W^{\dR}_{k+1} = \cdots = W^{\dR}_{2k+1} = H^1_{\dR,\cusp}(\mathcal{Y}_0(N),\mathcal{V}_k)   \\
&W^{\dR}_{2k+2} = H^1_{\dR}(\mathcal{Y}_0(N),\mathcal{V}_k)
\end{align*}
and
\begin{align*}
  &F_{\dR}^0 = H_{\dR}^1(\mathcal{Y}_0(N),\mathcal{V}_k)\\
  &F^1_{\dR} = \cdots = F_{\dR}^{k+1} = \im (H^0(\mathcal{X}_0(N),\overline{\mathcal{F}}^{\tensor k}\tensor \Omega_{\mathcal{X}_0(N)/\mathbb{Q}}^1(\log \mathcal{Z}_N))\to H_{\dR}^1(\mathcal{Y}_0(N),\mathcal{V}_k))\\
  &F_{\dR}^{k+2}=0\text{.}
\end{align*}
The weight filtration on $H^1(\mathcal{Y}_0(N)_{\mathbb{C}}^{\an},\mathbb{V}_k)$ is defined by
  \begin{align*}
    &W^{\B}_k=0\\
    &W^{\B}_{k+1} = \cdots = W^{\B}_{2k+1} = \im(H^1(\mathcal{X}_0(N)_{\mathbb{C}}^{\an}, j^{\an}_{\mathbb{C},*}\mathbb{V}_k)\to H^1(\mathcal{Y}_0(N)_{\mathbb{C}}^{\an},\mathbb{V}_k))\\
    &W^{\B}_{2k+2} = H^1(\mathcal{Y}_0(N)_{\mathbb{C}}^{\an},\mathbb{V}_k)\text{,}
  \end{align*}
where $j: \mathcal{Y}_0(N) \to \mathcal{X}_0(N)$ is the natural open immersion. The real Frobenius $F_{\infty}$ is induced by the complex conjugation $\mathcal{Y}_0(N)_{\mathbb{C}}^{\an} \to \mathcal{Y}_0(N)_{\mathbb{C}}^{\an}$; under the identification (\ref{eq:group-cohom}), we have
$$
F_{\infty}([C_f]) = [\gamma\mapsto C_f(\epsilon \gamma \epsilon^{-1})|_{\epsilon}]\text{, }\qquad \epsilon = \left(\begin{array}{cc} 1 & 0 \\ 0 &-1\end{array}\right) \text{.}
$$

\begin{obs}
Under the identification  (\ref{eq:group-cohom}), the subspace  $W^{\B}_{k+1}$ is given by the parabolic cohomology $\bigcap_{c}\ker (H^1(\Gamma_0(N),V_k) \to H^1(\Gamma_0(N)_c,V_k))$, where $c$ runs through the set of cusps of $\Gamma_0(N)$ and $\Gamma_0(N)_c\le \Gamma_0(N)$ denotes the stabiliser of the cusp $c$ (see \cite{zucker79} Proposition 12.5).
\end{obs}

\begin{theorem}\label{thm:object-H}
  With the above definitions, the triple
$$
H^1(\mathcal{Y}_0(N),V_k) \defeq ((H^1(\mathcal{Y}_0(N),V_k)_{\B}, W^{\B}, F_{\infty}),(H^1(\mathcal{Y}_0(N),V_k)_{\dR},W^{\dR},F_{\dR}),\comp)\text{,}
$$
is an object of the category of realisations $\mathcal{H}(\mathbb{Q})$. Moreover, the weight $k+1$ part of $H^1(\mathcal{Y}_0(N),V_k)$, which we denote by $H_{\cusp}^1(\mathcal{Y}_0(N),V_k)$, is of Hodge type $\{(k+1,0),(0,k+1)\}$.
\end{theorem}

Note that the Hodge filtration $F^{k+1}_{\dR}$ on $H^1_{\dR}(\mathcal{Y}_0(N),\mathcal{V}_k)$ (resp. on $H^1_{\dR,\cusp}(\mathcal{Y}_0(N),\mathcal{V}_k)$) is isomorphic to $M_{k+2}(\Gamma_0(N);\mathbb{Q})$ (resp. $S_{k+2}(\Gamma_0(N);\mathbb{Q})$) via the map $f\mapsto [f]$ of Theorem \ref{thm:derham-cohom}.

\begin{proof}
The commutativity of the diagram (\ref{diagram:comp-frobenius}) can be readily verified from the above explicit definitions of $\comp$ and $F_{\infty}$. That $(H^1(\mathcal{Y}_0(N),V_k)_{\B}, W^{\B}, \comp(F_{\dR}))$ is a $\mathbb{Q}$-mixed Hodge structure follows from a theorem of Zucker \cite{zucker79} (Sections 12 and 13). Finally, that $H_{\cusp}^1(\mathcal{Y}_0(N),V_k)$ is of Hodge type $\{(k+1,0),(0,k+1)\}$ is the content of the classical Eichler--Shimura isomorphism (see also \cite{zucker79} Section 12).
\end{proof}

The pure object $H^1_{\dR,\cusp}(\mathcal{Y}_0(N),\mathcal{V}_k)$ of $\mathcal{H}(\mathbb{Q})$ admits a canonical polarisation
$$
\langle \ , \ \rangle : H^1_{\cusp}(\mathcal{Y}_0(N),V_k) \times H^1_{\cusp}(\mathcal{Y}_0(N),V_k) \to \mathbb{Q}(-k-1)
$$
whose de Rham realisation $\langle \ , \ \rangle_{\dR}$ was defined at the end of Sect. \ref{sec:derham}. The Betti realisation $\langle \ , \ \rangle_{\B}$ is defined in terms of the cup product in the same way; here, we use the $\mathbb{Q}$-linear pairing of local systems $\langle \ , \ \rangle :\mathbb{V} \times \mathbb{V} \to \mathbb{Q}_{\mathcal{Y}_0(N)_{\mathbb{C}}^{\an}}$ whose fibres are induced by the usual intersection pairing on $H_1(E^{\an},\mathbb{Z})$, where $E$ is a complex elliptic curve. Explicitly, over $\mathbb{H}$,  we have
\begin{align}\label{eq:betti-pairing-XY}
\langle X , Y \rangle = -1\text{.}
\end{align}

Let us now consider the de Rham Hermitian form $( \ , \ )_{\dR}$ on $H^1_{\dR,\cusp}(\mathcal{Y}_0(N),\mathcal{V}_k)\tensor_{\mathbb{Q}}\mathbb{C}$ as defined in Sect. \ref{sec:sv-per-pol}. The next proposition shows that its restriction to $S_{k+2}(\Gamma_0(N))$ is, up to an explicit factor, the usual Petersson inner product of cusp forms.

\begin{prop}\label{prop:derham-petersson}
  Consider the injection $S_{k+2}(\Gamma_0(N)) \to H^1_{\dR,\cusp}(\mathcal{Y}_0(N),\mathcal{V}_k)\tensor_{\mathbb{Q}}\mathbb{C}$, $f \mapsto [f]$, of Sect. \ref{sec:derham}. For every $f,g \in S_{k+2}(\Gamma_0(N))$, we have
  $$
( [f] , [g] )_{\dR} = (4\pi)^{k+1}\int_{\Gamma_0(N)\backslash \mathbb{H}}f(\tau)\overline{g(\tau)} y^{k+2}\frac{dx\wedge dy}{y^2}
$$
where $\tau  = x + i y$.
\end{prop}

\begin{proof}
  We work in $C^{\infty}$ de Rham cohomology. Note that the real structure coming from Betti cohomology coincides with the real structure coming from real valued $C^{\infty}$ de Rham cohomology, so that, by Lemma \ref{lemma:betti-conj}, $\sv \tensor c_{\dR}$ acts as the usual complex conjugation of complex valued $C^{\infty}$ differential forms.  Since $[f]$ is represented by $f\omega^k\tensor (2\pi i d\tau)$, and similarly for $[g]$, it follows from the construction of the pairing $\langle \ , \ \rangle_{\dR}$ in terms of the de Rham cup product that
  $$
( [f] , [g] )_{\dR} = (-1)^{k+1}\frac{1}{2\pi i}\int_{\Gamma_0(N)\backslash \mathbb{H}}f(\tau)\overline{g(\tau)}\langle \omega^k, \overline{\omega}^k\rangle_k (2\pi i d\tau) \wedge (-2\pi i d \overline{\tau}).
$$
We compute the de Rham pairing $\langle \omega , \overline{\omega}\rangle$ in terms of the Betti pairing on $\mathbb{V}$ by using the relative comparison isomorphism (\ref{eq:rel-comp}) and Eq. (\ref{eq:betti-pairing-XY}):
\begin{align*}
  \langle \omega,\overline{\omega}\rangle = \frac{1}{2\pi i}\langle \comp(\omega), \overline{\comp(\omega}) \rangle   = \frac{1}{2\pi i}\langle 2\pi i (X-\tau Y),-2\pi i (X-\overline \tau Y) \rangle = -2\pi i (\overline{\tau} - \tau) =  -4\pi y
\end{align*}
Thus, $\langle \omega^k,\overline{\omega}^k\rangle_k = (-4\pi y)^k$ and the statement follows.
\end{proof}

\begin{obs}
Although we do not work explicitly with Betti realisations in this paper, let us remark that, under the identification (\ref{eq:group-cohom}), one can relate the Betti pairing $\langle \ , \ \rangle_{\B}$ with a generalisation of Harberland's inner product. Then, the Haberland formula (see \cite{PP13} Theorem 3.2), relating this inner product to the Petersson inner product, follows immediately from the commutativity of (\ref{eq:comm-pol}) and from the above proposition. 
\end{obs}

A classical computation relying on the orthonormality of $\{e^{2\pi i n x}\}_{n\in\mathbb{Z}}$ in $L^2([0,1])$ yields (cf. \cite{CS17} Theorem 8.2.3, \cite{rankin77} Theorem 5.7.3):

\begin{coro}\label{coro:derham-poincare}
  Let $m\ge 1$ be an integer, and $f \in S_{k+2}(\Gamma_0(N))$ be a cusp form. Then,
  $$
  ([f],[P_{m,k+2,N}])_{\dR} = \frac{k!}{m^{k+1}}a_m(f)\text{.}
  $$ \hfill $\square$
\end{coro}

Let us recall that the above formula implies in particular that the Poincare series $P_{m,k+2,N}$ with $m\ge 1$ generate the (finite-dimensional) vector space $S_{k+2}(\Gamma_0(N))$; this will be used throughout in Sect. \ref{sec:general-case}.

\section{Harmonic Maass forms and the single-valued involution}\label{sec:maass}

Next, we relate the single-valued involution on the objects $H^1(\mathcal{Y}_0(N),V_k)$ of $\mathcal{H}(\mathbb{Q})$ constructed in Sect. \ref{sec:real-modular-motives} with the theory of harmonic Maass forms of integral weight, which we now briefly recall.

The hyperbolic Laplacian of weight $r\in \mathbb{Z}$ is the differential operator on $\mathbb{H}$ defined by
$$
\Delta_r = - y\left(\frac{\partial^2}{\partial x^2} + \frac{\partial^2}{\partial y^2} \right) + iry\left(\frac{\partial}{\partial x} + i\frac{\partial}{\partial y} \right)\text{, }
$$
where $\tau = x+iy$.

\begin{defi}(cf. \cite{BFOR17} 4.1)
  A \emph{harmonic Maass form of manageable growth} of weight $r$ and level $\Gamma_0(N)$ is a $C^{\infty}$ function $F: \mathbb{H} \to \mathbb{C}$ which is modular of weight $r$ for $\Gamma_0(N)$:
  $$
F|_{\gamma,r}=F
$$
for every $\gamma \in \Gamma_0(N)$, harmonic:
$$
\Delta_rF = 0\text{,}
$$
and has `manageable growth' at all cusps: for every $g \in \SL_2(\mathbb{Z})$, there exists $\rho>0$ such that
$$
F|_{g,r} = O(e^{\rho \Im \tau})
$$
as $\Im \tau \rightarrow +\infty$.
\end{defi}

The space of harmonic Maass forms of manageable growth of weight $r$ and level $\Gamma_0(N)$ is denoted by $H_r^!(\Gamma_0(N))$. If the last condition above is replaced by ``for every $g \in \SL_2(\mathbb{Z})$, there exists $P \in \mathbb{C}[q^{-1}]$ and $\rho >0$ such that
$$
F|_{g,r}-P(e^{2\pi i \tau}) = O(e^{-\rho \Im \tau})
$$
as $\Im \tau \rightarrow +\infty$'', then we say that $F$ is a \emph{harmonic Maass form} of weight $r$ and level $\Gamma_0(N)$. The subspace of harmonic Maass forms is denoted by $H_r(\Gamma_0(N))\subset H_r^!(\Gamma_0(N))$; it contains the space $M^!_{r}(\Gamma_0(N))$ of weakly holomorphic modular forms of weight $r$ and level $\Gamma_0(N)$.

The following result summarises the main properties of harmonic Maass forms of manageable growth with respect to the differential operators $\xi_{-k}$ and $D^{k+1}$, defined on a smooth function $F:\mathbb{H} \to \mathbb{C}$ by 
$$
\xi_{-k}(F) = 2i (\Im \tau)^{-k} \overline{\frac{\partial F}{\partial \overline{\tau}}} \text{, }\ \ \ D^{k+1}(F) = \frac{1}{(2\pi i)^{k+1}}\frac{\partial^{k+1}F}{\partial\tau^{k+1}}\text{.}
$$

\begin{theorem}[\cite{BF04}, \cite{BOR08}, \cite{BFOR17} Chapter 5]
  Let $k \ge 0$  be an even integer, and $F \in H_{-k}^!(\Gamma_0(N))$. Then $\xi_{-k}F$ and $D^{k+1}F$ belong to $M_{k+2}^!(\Gamma_0(N))$. Moreover, $F \in H_{-k}(\Gamma_0(N))$ if and only if $\xi_{-k}F \in S_{k+2}(\Gamma_0(N))$. Finally, the following sequence of $\mathbb{C}$-vector spaces
  $$
    \begin{tikzcd}
      0 \arrow{r} & M_{-k}^!(\Gamma_0(N)) \arrow{r} & H_{-k}^!(\Gamma_0(N)) \arrow{r}{\xi_{-k}} &  M_{k+2}^!(\Gamma_0(N)) \arrow{r} & 0
  \end{tikzcd}
   $$
  is exact.
\end{theorem}

To state our next theorem, we introduce the following notion, which will also be used in the following sections.

\begin{defi}\label{defi:betti-conjugate}
  We say that $f,g \in M_{k+2}^!(\Gamma_0(N))$ are \emph{Betti conjugate} if
  $$
  (\sv\tensor c_{\dR})([f]) = [g]
  $$
  in $H^1_{\dR}(\mathcal{Y}_0(N),\mathcal{V}_k)\tensor_{\mathbb{Q}}\mathbb{C}$.
\end{defi}

Our terminology is justified by Lemma \ref{lemma:betti-conj}, which implies that $f$ and $g$ are Betti conjugate if and only if their classes in Betti cohomology are complex conjugate:
$$
(\id \tensor c_{\B})(\comp([f])) = \comp([g])\text{.}
$$

\begin{theorem}\label{thm:maass-forms}
  Let $k\ge 0$ be an even integer, and $f,g \in M_{k+2}^!(\Gamma_0(N))$ be Betti conjugate. Then there exists $F \in H_{-k}^!(\Gamma_0(N))$ such that
  $$
\xi_{-k}F = (4\pi)^{k+1}f\ \ \text{ and }\ \ D^{k+1}F = k! g\text{.}
$$
If $k=0$, then $F$ is unique up to an additive constant; otherwise, $F$ is unique.
\end{theorem}

This theorem is essentially equivalent to the results of Brown \cite{brown18} when $N=1$, and of Candelori \cite{candelori14} when $N\ge 5$ and $f$ is a cusp form. We prove it below for completeness; our approach is similar to Candelori's.

\begin{proof}
  Recall that working over $\mathcal{Y}_0(N)_{\mathbb{C}}^{\an}$ amounts to working over $\mathbb{H}$ equivariantly under the left action of $\Gamma_0(N)$. This will be implicit in what follows.

  After the Kodaira--Spencer identification $\mathcal{F}^{\tensor 2} \cong \Omega^1_{\mathcal{Y}_0(N)/\mathbb{Q}}$, we have
  $$
\omega_f = f\omega^k \tensor 2\pi i\, d\tau \in H^0(\mathcal{Y}_0(N)_{\mathbb{C}}^{\an}, \mathcal{V}^{\an}_k\tensor \Omega^1_{\mathcal{Y}_0(N)_{\mathbb{C}}^{\an}})\text{,}
$$
and similarly for $\omega_g$. As the real structure on the complex analytic de Rham cohomology $H^1_{\dR}(\mathcal{Y}_0(N)_{\mathbb{C}}^{\an},\mathcal{V}_{k}^{\an})$ coming from real valued $C^{\infty}$ de Rham cohomology  coincides with the real structure given by the Betti cohomology $H^1(\mathcal{Y}_0(N)_{\mathbb{C}}^{\an},\mathbb{V}_k)\tensor_{\mathbb{Q}} \mathbb{R}$ after the natural comparison isomorphisms (cf. proof of Proposition \ref{prop:derham-petersson}), the image of the $C^{\infty}$ differential form with coefficients in $\mathcal{V}^{\an}_k$
$$
\omega_{g}-\overline{\omega_f} = g\omega^k\tensor 2\pi i\, d\tau + \overline{f}\overline{\omega}^k\tensor 2\pi i\, d\overline{\tau}
$$
in $H^1_{\dR}(\mathcal{Y}_0(N)_{\mathbb{C}}^{\an}, \mathcal{V}_k^{\an})$ is zero by Lemma \ref{lemma:betti-conj}. Thus, there exists a $C^{\infty}$ global section $A$ of $\mathcal{V}_k^{\an}$ satisfying
$$
\nabla A =  \omega_{g}-\overline{\omega_f} \text{,}
$$
where $\nabla$ denotes the Gauss--Manin connection on $\mathcal{V}_k^{\an}$.  Since $\omega$ and $\overline{\omega}$ trivialise $\mathcal{V}^{\an}$ over $\mathbb{H}$, we can uniquely write
$$
A = \sum_{r+s=k}A_{r,s} \omega^r \overline{\omega}^s\text{,}
$$
where $A_{r,s}$ are $C^{\infty}$ functions on $\mathbb{H}$. We set
$$
F \defeq (-4\pi i \Im \tau )^k A_{0,k}\text{.}
$$

To check that $F$ satisfies $\xi_{-k}F = (4\pi)^{k+1}f$ and $D^{k+1}F=k!g$, it is convenient to use Brown's terminology and notation in \cite{brown18}. Note first that, by $\Gamma_0(N)$-invariance of $A$, the $C^{\infty}$ functions $A_{r,s}$ are modular of weights $(r,s)$ for $\Gamma_0(N)$. In particular, since
$$
2\mathbb{L} \defeq -4\pi \Im \tau
$$
is modular of weights $(-1,-1)$, the function $F$ is modular of weights $(-k,0)$.

By definition of $A$, we have
\begin{align}\label{eq:nablaDA}
\nabla_DA = g\omega^k\text{.}
\end{align}
Using the formulas
$$
\nabla_D\omega = \frac{\omega + \overline{\omega}}{2\mathbb{L}}\text{, } \ \ \ \nabla_{D}\overline{\omega} = 0\text{,}
$$
which can be checked by computing the periods along the 1-cycles $\sigma_1$ and $\sigma_\tau$, we obtain from (\ref{eq:nablaDA}) the equations
$$
\partial_k A_{k,0} = 2\mathbb{L}g\ \ \ \text{ and }\ \ \ \partial_rA_{r,k-r} + (r+1)A_{r+1,k-r-1} = 0\text{, }\qquad 0\le r < k\text{,}
$$
where $\partial_r \defeq 2\mathbb{L}D + r$. Thus, (\cite{brown18} Lemmas 3.2 and 3.3)
$$
D^{k+1}F = \frac{1}{(2\mathbb{L})^{k+1}}\partial_0 \partial_{-1} \cdots \partial_{-k}F = \partial_k \partial_{k-1}\cdots \partial_0 A_{0,k} = (-1)^kk!g = k!g\text{.}
$$
Using that
$$
\nabla_{\overline{D}}A = -\overline{f} \overline{\omega}^k
$$
we conclude similarly that $\xi_{-k}F = (4\pi)^{k+1}f$.

Finally, we check that $F$ is a harmonic Maass form of manageable growth of weight $-k$ and level $\Gamma_0(N)$. The modularity property has already been remarked above. That $F$ is harmonic follows from the formula $\Delta_{-k} = -\partial_{-k-1}\overline{\partial}_0$ and from the fact that $\xi_{-k}F = (4\pi)^{k+1}f$ is holomorphic. The growth condition at the cusps follows from those of $f$ and $g$, and from the equations $D^{k+1}F = k! g$ and $\xi_{-k}F = (4\pi)^{k+1}f$; this can be seen either via the Fourier expansion of $F$ as in \cite{BF04} (Section 3), or via an explicit expression of $F$ in terms of $f$ and $g$ as in \cite{brown18} (Section 5.3).
\end{proof}

\begin{coro}\label{coro:maass-forms}
  The following diagram of $\mathbb{C}$-vector spaces commutes:
  $$
  \begin{tikzcd}[column sep=small]
    & H_{-k}^!(\Gamma_0(N))\arrow{ld}[swap]{\frac{1}{(4\pi)^{k+1}}\xi_{-k}}\arrow{rd}{\frac{1}{k!}D^{k+1}} &\\
    M_{k+2}^!(\Gamma_0(N)) \arrow{d} &  & M_{k+2}^!(\Gamma_0(N)) \arrow{d}\\
    H^1_{\dR}(\mathcal{Y}_0(N),\mathcal{V}_k)\tensor_{\mathbb{Q}} \mathbb{C} \arrow{rr}[swap]{\sv\tensor c_{\dR}} & & H^1_{\dR}(\mathcal{Y}_0(N),\mathcal{V}_k)\tensor_{\mathbb{Q}} \mathbb{C}
  \end{tikzcd}
  $$
  where the vertical arrows are the quotient maps $f\mapsto [f]$ of Theorem \ref{thm:derham-cohom}. \hfill $\square$
\end{coro}

\begin{obs}\label{obs:regularized}
  In particular, we obtain a formula for the de Rham pairing $( \ , \ )_{\dR}$ defined in Lemma \ref{lemma:derham-hermitian} in terms of harmonic lifts.  Namely, let $f,g \in M_{k+2}^{!,\infty}(\Gamma_0(N))$, and consider any $G \in H^{!}_{-k}(\Gamma_0(N))$ such that $\xi_{-k}G = (4\pi)^{k+1}g$. Then, by (\ref{eq:derham-cup})
  \begin{align*}
  ([f],[g])_{\dR} &= -\langle [f], (\sv \tensor c_{\dR})([g])\rangle_{\dR} \\
  &= -\frac{1}{k!}\langle [f], [D^{k+1}G]\rangle_{\dR} = \sum_{n\in \mathbb{Z}}\frac{a_n(f)a_{-n}(D^{k+1}G)}{n^{k+1}}= \sum_{n\in \mathbb{Z}}a_n(f)a_{-n}(G^+)
  \end{align*}
  where $G^+$ is the holomorphic part of $G$ (see \cite{BFOR17} Section 4.2). This provides a natural cohomological interpretation for the `regularised Petersson inner product' (cf. \cite{BDE17} Theorem 1.1).
\end{obs}

As an application of Corollary \ref{coro:maass-forms}, we have the following result on Poincaré series.

\begin{prop}\label{prop:poincare-betti-conjugate}
Let $k\ge 0$ be an even integer. For every integer $m>0$, the Poincaré series $P_{m,k+2,N}$ and $-P_{-m,k+2,N}$ are Betti conjugate.
\end{prop}

Actually, since Poincaré series have real Fourier coefficients at infinity (Proposition \ref{prop:poincare}), we have
  $$
\sv([P_{m,k+2,N}]) = - [P_{-m,k+2,N}]
$$
in $H^1_{\dR}(\mathcal{Y}_0(N), \mathcal{V}_k)\tensor_{\mathbb{Q}}\mathbb{R}$. We remark that, in level 1, this is consistent with the action of $\sv$ on Eisenstein series in \cite{brown18}, which corresponds to the case $m=0$ in our notation.

\begin{proof}
We use the following result of Bringmann and Ono \cite{BO07} (see also \cite{BOR08}): there exists a harmonic Maass form $Q_{-m,k+2,N} \in H^!_{-k}(\Gamma_0(N))$, the so-called \emph{Maass--Poincaré series}, satisfying
$$
\xi_{-k}Q_{-m,k+2,N}= \frac{(4\pi)^{k+1}m^{k+1}}{k!}P_{m,k+2,N}\ \ \ \text{ and }\ \ \ D^{k+1}Q_{-m,k+2,N} = -m^{k+1}P_{-m,k+2,N}\text{.}
$$
By considering $\frac{k!}{m^{k+1}}Q_{-m,k+2,N}\in H^!_{-k}(\Gamma_0(N))$, our statement follows immediately from Corollary \ref{coro:maass-forms}.
\end{proof}

\section{Coefficients of Poincaré series as single-valued periods: rank 2 case}\label{sec:rank2}

Let $k\ge 0$ and $N\ge 1$ be integers such that $\dim S_{k+2}(\Gamma_0(N))=1$, so that
$$
\dim H^1_{\dR,\cusp}(\mathcal{Y}_0(N),\mathcal{V}_k) = 2\text{.}
$$
This only happens in the finite number of cases given by the tables below (see \cite{LMFDB}):

\vspace{0.2cm}

\begin{center}
 \begin{tabular}{|c | c|} 
   \hline  
 $N$ & $k+2$  \\ [0.1ex] 
 \hline\hline
 1 & 12, 16, 18, 20, 22, 26  \\ 
 \hline
 2 & 8, 10   \\
 \hline
 3 & 6, 8   \\
 \hline
 4 & 6   \\
 \hline
   5 & 4, 6  \\
   \hline
   6 & 4  \\
   \hline
   7 & 4  \\   
 \hline
\end{tabular} \ \ \ \begin{tabular}{|c | c| } 
   \hline  
 $N$ & $k+2$  \\ [0.1ex] 
                       \hline\hline
                         8 & 4 \\
   \hline
                       9 & 4 (CM) \\
                       \hline
 11 & 2  \\ 
 \hline
 14 & 2  \\
 \hline
 15 & 2  \\
 \hline
 17 & 2  \\
 \hline
   19 & 2 \\
   \hline
                     \end{tabular}\ \ \ 
 \begin{tabular}{|c | c|} 
   \hline  
 $N$ & $k+2$  \\ [0.1ex] 
                       \hline\hline
                      20 & 2 \\
   \hline
   21 & 2 \\
   \hline
  24 & 2 \\
   \hline
   27 & 2 (CM) \\   
   \hline
   32 & 2 (CM) \\   
   \hline
   36 & 2 (CM) \\   
   \hline
   49 & 2 (CM) \\   
   \hline
 \end{tabular}
\end{center}

\vspace{0.2cm}

We now proceed to a proof of our main theorem relating single-valued periods to coefficients of Poincaré series in this particular case. This serves as an illustration of our proof method for the general case.

Let $f \in S_{k+2}(\Gamma_0(N);\mathbb{Q})$ and $g \in S^{!,\infty}_{k+2}(\Gamma_0(N);\mathbb{Q})$ be such that $([f],[g])$ is a $\mathbb{Q}$-basis of $H^1_{\dR,\cusp}(\mathcal{Y}_0(N),\mathcal{V}_k)$  satisfying
$$
\langle [f], [g] \rangle_{\dR} = 1\text{,}
$$
and write
$$
\sv (\begin{array}{cc} [f] & [g] \end{array}) = (\begin{array}{cc} [f] & [g] \end{array})\cdot \left(\begin{array}{cc} a & b \\ c & d \end{array} \right) 
$$
for some $a,b,c,d \in \mathbb{R}$. These are the single-valued periods of $H^1_{\cusp}(\mathcal{Y}_0(N),V_k)$; recall from Proposition \ref{prop:svp-matrix-polarisation} that $c\neq 0$, $ad-bc=-1$, and $d=-a$. Moreover, we have
$$
([g],[f])_{\dR} = a\ \ \ \text{ and }\ \ \  ([f],[f])_{\dR} = -c\text{,}
$$
where $( \ , \ )_{\dR}$ is the de Rham Hermitian form on $H^1_{\dR,\cusp}(\mathcal{Y}_0(N),\mathcal{V}_k)\tensor_{\mathbb{Q}}\mathbb{C}$, defined in Lemma \ref{lemma:derham-hermitian} by $(\omega,\eta)_{\dR} = -\langle \omega, (\sv\tensor c_{\dR})(\eta)\rangle_{\dR}$; see formulas (\ref{eq:dR-form-sv-period}) and (\ref{eq:dR-form-sv-period2}). Set
$$
\rho \defeq \frac{a}{c} = -\frac{([g],[f])_{\dR}}{([f],[f])_{\dR}} \in \mathbb{R}\text{.}
$$

From now on, for simplicity, we denote
$$
P_{m}\defeq P_{m,k+2,N} \in S_{k+2}^{!,\infty}(\Gamma_0(N);\mathbb{Q})\text{.}
$$

\begin{prop}\label{prop:sv-periods-rank-2}
  For every positive integer $m$, there exists $h_m \in M_{-k}^{!,\infty}(\Gamma_0(N);\mathbb{Q})$, depending on $f$ and $g$, such that, for every integer $n\ge 1$,
  $$
a_n(P_m) = -\frac{k!}{m^{k+1}}a_m(f)a_n(f)c^{-1}\ \ \text{ and }\ \ a_n(P_{-m}) =  \frac{k!}{m^{k+1}}a_m(f)a_n(f)\rho + r_{m,n} \text{,}
$$
where
$$
r_{m,n} = \frac{k!}{m^{k+1}}a_m(f)a_n(g) + n^{k+1}a_n(h_m) \in \mathbb{Q}\text{.}
$$
In particular, $P_m$ vanishes identically if and only if $a_m(f)=0$, and in this case $a_n(P_{-m})\in \mathbb{Q}$ for every $n\ge 0$.
\end{prop}

It also follows from the above proposition that $\mathbb{Q}(\sv)\defeq \mathbb{Q}(a,b,c,d)= \mathbb{Q}(c,\rho)$ coincides with the field $\mathbb{Q}(a_n(P_m)\text{ ; }n\ge 1\text{, }m\neq 0)$ generated by the coefficients of all the weakly holomorphic Poincaré series of level $N$ and weight $k+2$.

\begin{proof}
Let $m$ be a positive integer and, since $\dim S_{k+2}(\Gamma_0(N))=1$, write
\begin{align}\label{eq:Pm=lambdaf}
P_m = \lambda f
\end{align}
for some $\lambda \in \mathbb{R}$. By Corollary \ref{coro:derham-poincare}, we have
$$
\frac{k!}{m^{k+1}}a_m(f) = ([f],[P_m])_{\dR} = \lambda ([f],[f])_{\dR} = -\lambda c
$$
so that
\begin{align}\label{eq:lambda}
\lambda = -\frac{k!}{m^{k+1}}a_m(f)c^{-1}\text{.}
\end{align}
In particular, we obtain the following general formula for the Fourier coefficients at infinity of $P_m$ in terms of the single-valued period $c$:
$$
a_n(P_m) = -\frac{k!}{m^{k+1}}a_m(f) a_n(f)c^{-1}\text{.}
$$
Note that $P_m$ vanishes identically if and only if $a_m(f)=0$.

Now, set
\begin{equation}\label{eq:f-flat}
f^{\flat} \defeq af + cg \in S_{k+2}^{!,\infty}(\Gamma_0(N);\mathbb{R})\text{.}
\end{equation}
Since $f^{\flat}$ and $f$ are Betti conjugate (Definition \ref{defi:betti-conjugate}), and $P_{-m}$ and $-P_m$ are Betti conjugate (Proposition \ref{prop:poincare-betti-conjugate}), it follows from (\ref{eq:Pm=lambdaf}) that $-\lambda f^{\flat}$ and $P_{-m}$ map to the same class in
$$
H^1_{\dR,\cusp}(\mathcal{Y}_0(N),\mathcal{V}_k)\tensor_{\mathbb{Q}}\mathbb{R} \cong S^{!,\infty}_{k+2}(\Gamma_0(N);\mathbb{R})/ D^{k+1}M_{-k}^{!,\infty}(\Gamma_0(N);\mathbb{R})\text{.}
$$
Thus, there exists $h_m\in M_{-k}^{!,\infty}(\Gamma_0(N);\mathbb{R})$ such that
\begin{align}\label{eq:P-m=-lambdafflat}
P_{-m} = -\lambda f^{\flat} + D^{k+1}h_m\text{.}
\end{align}
By considering principal parts at the cusp at infinity, we get
\begin{align}\label{eq:principal-part-h}
\frac{1}{q^m} = -\lambda \mathcal{P}_{f^{\flat}} + \mathcal{P}_{D^{k+1}h_m} \stackrel{(\ref{eq:f-flat})}{=} -\lambda c\mathcal{P}_g + D^{k+1}\mathcal{P}_{h_m} =  \frac{k!}{m^{k+1}}a_m(f)\mathcal{P}_g + D^{k+1}\mathcal{P}_{h_m}\text{.}
\end{align}
As both $f$ and $g$ have rational Fourier coefficients at infinity, we get $D^{k+1}\mathcal{P}_{h_m} \in \mathbb{Q}[q^{-1}]$. Thus, $\mathcal{P}_{h_m} \in \mathbb{Q}[q^{-1}]$, and Lemma \ref{lemma:principal-part} implies that $h_m \in M_{-k}^{!,\infty}(\Gamma_0(N);\mathbb{Q})$.

Finally, by taking Fourier coefficients at infinity in (\ref{eq:P-m=-lambdafflat}) and applying (\ref{eq:lambda}), we get
 $$
 a_n(P_{-m}) =  \frac{k!}{m^{k+1}}a_m(f)a_n(f)\rho + \frac{k!}{m^{k+1}}a_m(f)a_n(g) + n^{k+1}a_n(h_m)\text{.}
 $$
\end{proof}

When $m=1$, we can take $h_m$ to be zero. Indeed, in the above cases we always have $a_1(f)\neq 0$, so we can normalise $a_1(f)=1$. Moreover, we can assume that $\ord_{\infty}(g)\ge -1$, so that $\mathcal{P}_g = (k!)^{-1}q^{-1}$ by the condition $\langle [f],[g] \rangle_{\dR}=1$. Eq. (\ref{eq:principal-part-h}) in the above proof then implies that $\mathcal{P}_{h_1} = 0$, which yields $h_1=0$.

\begin{coro}
  For $f$ and $g$ as above, we have $a_n(P_{-1}) = k!(a_n(f)\rho + a_n(g))$. \hfill $\square$  
\end{coro}

When $N=1$ and $k+2=12$, this recovers Brown's formula (\cite{brown18} Corollary 1.4), up to a difference in normalisation. 
 
\section{Coefficients of Poincaré series as single-valued periods: general case}\label{sec:general-case}

From now on, we fix an even integer $k\ge 0$, and an integer $N\ge 1$. For simplicity, we denote
$$
P_m \defeq P_{m,k+2,N} \in S^{!,\infty}_{k+2}(\Gamma_0(N);\mathbb{R})
$$
for every $m \in \mathbb{Z}\setminus\{0\}$.

\begin{prop}\label{prop:poincare-generators}
  Let $m_1,\dots,m_s \ge 1$ be integers such that $P_{m_1},\ldots,P_{m_s}$ generate the $\mathbb{R}$-vector space $S_{k+2}(\Gamma_0(N);\mathbb{R})$. Then, for every integer $m\ge 1$, there exist $\lambda_1,\ldots,\lambda_s \in \mathbb{Q}$ and $h \in M_{-k}^{!,\infty}(\Gamma_0(N);\mathbb{Q})$ such that
  $$
  P_m = \lambda_1P_{m_1} + \cdots + \lambda_s P_{m_s}
  $$
  and
  $$
P_{-m} = \lambda_1 P_{-m_1} + \cdots + \lambda_s P_{-m_s} + D^{k+1}h\text{.}
  $$
\end{prop}

\begin{proof}
 For every integer $m\ge 1$, recall that we view $f\mapsto a_m(f)$ as a linear functional on $S_{k+2}(\Gamma_0(N))$; it follows from Corollary \ref{coro:derham-poincare} that
 \begin{align}\label{eq:petersson-poincare}
( \ \ , [P_m])_{\dR} = \frac{k!}{m^{k+1}}a_m
 \end{align}
 in $S_{k+2}(\Gamma_0(N))^{\vee}$. Since $P_{m_1},\ldots,P_{m_s}$ generate $S_{k+2}(\Gamma_0(N);\mathbb{R})$, and since $( \ , \ )_{\dR}$ is non-degenerate on $S_{k+2}(\Gamma_0(N))$, the functionals  $a_{m_1},\ldots,a_{m_s}$ generate the $\mathbb{Q}$-vector space $S_{k+2}(\Gamma_0(N);\mathbb{Q})^{\vee}$. Thus, there exist $\lambda_1,\ldots,\lambda_s \in \mathbb{Q}$ such that $a_m = \lambda_1 a_{m_1} + \cdots + \lambda_s a_{m_s}$. Again, by (\ref{eq:petersson-poincare}) and by the non-degeneracy of $( \ , \ )_{\dR}$, we conclude that
$$
P_m = \lambda_1 P_{m_1} + \cdots + \lambda_s P_{m_s}\text{.}
$$

Now, Proposition \ref{prop:poincare-betti-conjugate} implies that both $P_{-m}$ and $\lambda_1P_{-m_1} + \cdots + \lambda_s P_{-m_s}$ are Betti conjugate to $-P_m$; in particular, they map to the same class in $H^1_{\dR}(\mathcal{Y}_0(N),\mathcal{V}_k)\tensor_{\mathbb{Q}}\mathbb{R}$. Thus, by Corollary \ref{coro:derahm-whmf}, there exists $h\in M^{!,\infty}_{-k}(\Gamma_0(N);\mathbb{R})$ such that
$$
P_{-m} = \lambda_1P_{-m_1} + \cdots + \lambda_s P_{-m_s} + D^{k+1}h\text{.}
$$
Finally, since $\lambda_i \in \mathbb{Q}$, by taking principal parts at the cusp at infinity, we deduce from Proposition \ref{prop:poincare} and Lemma \ref{lemma:principal-part} that $h \in M_{-k}^{!,\infty}(\Gamma_0(N);\mathbb{\mathbb{Q}})$.
\end{proof}

\begin{coro}\label{coro:generator-poincare}
  Let $m_1,\dots,m_s \ge 1$ be integers such that $P_{m_1},\ldots,P_{m_s}$ generate the $\mathbb{R}$-vector space $S_{k+2}(\Gamma_0(N);\mathbb{R})$. Then, for every integers $m,n\ge 1$, 
  \begin{enumerate}
  \item $a_n(P_{m})$ lies in the $\mathbb{Q}$-linear span of $\{a_n(P_{m_i}) \text{ ; } 1\le i \le s\}$, and
  \item $a_n(P_{-m})$ lies in the $\mathbb{Q}$-linear span of $\{1\}\cup \{a_n(P_{-m_i}) \text{ ; } 1 \le i \le s\}$. 
    \end{enumerate}
  \end{coro} 

  \begin{proof}
Consider Fourier coefficients at infinity in the statement of Proposition \ref{prop:poincare-generators}.
  \end{proof}
  
We come to our main theorem.

\begin{theorem}\label{thm:main}
  Let $\mathbb{Q}(\sv)\subset \mathbb{R}$ be the field of rationality of the single-valued involution
  $$
\sv: H^1_{\dR,\cusp}(\mathcal{Y}_0(N),\mathcal{V}_k)\tensor_{\mathbb{Q}}\mathbb{R} \to H^1_{\dR,\cusp}(\mathcal{Y}_0(N),\mathcal{V}_k)\tensor_{\mathbb{Q}}\mathbb{R}\text{,}
$$
and $\mathbb{Q}(P)\subset \mathbb{R}$ be the $\mathbb{Q}$-extension generated by the Fourier coefficients at infinity of all the Poincaré series $P_m$, for $m\in \mathbb{Z}\setminus\{0\}$. Then
$$
\mathbb{Q}(\sv) = \mathbb{Q}(P)\text{.}
$$
\end{theorem}

\begin{proof}
  Let $d = \dim S_{k+2}(\Gamma_0(N))$. We showed in Sect. \ref{sec:real-modular-motives} that the object $H_{\cusp}^1(\mathcal{Y}_0(N),V_k)$ of $\mathcal{H}(\mathbb{Q})$ is pure of Hodge type $\{(k+1,0),(0,k+1)\}$ and admits a polarisation $\langle \ , \ \rangle$. In particular, by Corollary \ref{coro:derahm-whmf}, there exist $f_1,\ldots,f_d,g_1,\ldots,g_d \in S_{k+2}^{!,\infty}(\Gamma_0(N);\mathbb{Q})$ such that
  $$
b_{\dR} = ([f_1],\ldots,[f_d],[g_1],\ldots,[g_d])
$$
is a $\mathbb{Q}$-basis of $H^1_{\dR,\cusp}(\mathcal{Y}_0(N),\mathcal{V}_k)$ as in (\ref{eq:symplectic-Hodge-basis}), that is,
\begin{itemize}
\item $([f_1],\ldots,[f_d])$ is a $\mathbb{Q}$-basis of $F_{\dR}^{k+1}H^1_{\dR,\cusp}(\mathcal{Y}_0(N),\mathcal{V}_k)$ (or, equivalently, $(f_1,\ldots,f_d)$ is a $\mathbb{Q}$-basis of $S_{k+2}(\Gamma_0(N);\mathbb{Q})$), and
\item $\langle [f_i], [g_j] \rangle_{\dR} = \delta_{ij}$ and $\langle [g_i],[g_j] \rangle_{\dR} = 0$ for every $1\le i,j\le d$.
\end{itemize}

We prove first that $\mathbb{Q}(P)\subset \mathbb{Q}(\sv)$. For this, let $m_1,\ldots,m_d\ge 1$ be integers such that $(P_{m_1},\ldots,P_{m_d})$ is an $\mathbb{R}$-basis of $S_{k+2}(\Gamma_0(N);\mathbb{R})$, and write
   $$
   S = \left(\begin{array}{cc}
               A & B \\
               C & D
             \end{array}\right) \in {\GL}_{2g}(\mathbb{R})
           $$
           for the matrix of the single-valued involution $\sv$ in the basis $b_{\dR}$. By Corollary \ref{coro:generator-poincare}, it suffices to prove that $a_n(P_{ m_i}), a_n(P_{-m_i})\in \mathbb{Q}(S)$ for every integer $n\ge 1$ and $1\le i \le d$. Let $\Lambda = (\Lambda_{ij})_{i,j} \in {\GL}_d(\mathbb{R})$ be such that
\begin{align}\label{eq:poincare-base-change}
P_{m_j} = \sum_{i=1}^d\Lambda_{ij}f_i\text{, }\qquad 1\le j \le d\text{.}
\end{align}
By applying Corollary \ref{coro:derham-poincare} in the above formula,  we obtain
$$
\frac{k!}{m_j^{k+1}}a_{m_j}(f_r) = ([f_r],[P_{m_j}])_{\dR} = \sum_{i=1}^d([f_r],[f_i])_{\dR}\Lambda_{ij} = -\sum_{i=1}^dC_{ri}\Lambda_{ij}\text{,}
$$
where we used the identity (\ref{eq:dR-form-sv-period}) in the last equality above. Set $Q \defeq  (\frac{k!}{m_j^{k+1}}a_{m_j}(f_r))_{r,j} \in \GL_{d}(\mathbb{Q})$. The above equations amount to the matrix identity in $\GL_d(\mathbb{R})$
\begin{align}\label{eq:lambda-and-c}
  \Lambda = -C^{-1}Q\text{.}
\end{align}
Since $Q$ has rational coefficients, this proves already that $a_n(P_{m_j}) \in \mathbb{Q}(C) \subset \mathbb{Q}(S)$ for every $n\ge 1$ and $1\le j \le d$.

For every $1\le j \le d$, set
\begin{align*}
f^{\flat}_j \defeq \sum^d_{i=1}A_{ij}f_i + C_{ij}g_i \in S^{!,\infty}_{k+2}(\Gamma_0(N);\mathbb{Q}(\sv))\text{.}
\end{align*}
By definition of $S$, $f_j^{\flat}$ is Betti conjugate to $f_j$. Then, it follows from Proposition \ref{prop:poincare-betti-conjugate} and from (\ref{eq:poincare-base-change}) that both $P_{-m_j}$ and $- \sum_{i=1}^d \Lambda_{ij}f^{\flat}_i$ are Betti conjugate to $-P_{m_j}$. Thus, by Corollary \ref{coro:derahm-whmf}, there exists $h_j \in M_{-k}^{!,\infty}(\Gamma_0(N))$ such that
\begin{align}\label{eq:negative-m}
P_{-m_j} = - \sum_{i=1}^d \Lambda_{ij}f^{\flat}_i + D^{k+1}h_j\text{.}
\end{align}
By taking principal parts at the cusp at infinity in the above equation, and by applying Lemma \ref{lemma:principal-part} and (\ref{eq:lambda-and-c}), we get
$$
h_j \in M_{-k}^{!,\infty}(\Gamma_0(N); \mathbb{Q}(\Lambda))= M_{-k}^{!,\infty}(\Gamma_0(N); \mathbb{Q}(C))\text{.}
$$
Finally, by taking Fourier coefficients at infinity in Eq. (\ref{eq:negative-m}), we obtain
$$
a_n(P_{-m_j}) \in \mathbb{Q}(A,C) \subset \mathbb{Q}(S)
$$
for every integer $n\ge 1$. This concludes the proof that $\mathbb{Q}(P)\subset \mathbb{Q}(S)$.

For the reverse inclusion, note that Proposition \ref{prop:svp-matrix-polarisation} yields $\mathbb{Q}(\sv) = \mathbb{Q}(A,C)$. The formulas (\ref{eq:poincare-base-change}) and (\ref{eq:lambda-and-c}) above imply that $\mathbb{Q}(C)\subset \mathbb{Q}(P)$; thus, to finish our proof, it suffices to show that $\mathbb{Q}(A) \subset \mathbb{Q}(P)$.

If we denote $\Lambda' = \Lambda^{-1} \in \GL_{d}(\mathbb{Q}(C))$, then we can write
$$
f_{j} = \sum_{i=1}^d\Lambda'_{ij}P_{m_i}\text{,}\qquad 1\le j \le d\text{.}
$$
Arguing as above, we deduce that both $f^{\flat}_j$ and $-\sum_{i=1}^d \Lambda_{ij}'P_{-m_i}$ are Betti conjugate to $f_j$, so that there exists $h_j' \in M_{-k}^{!,\infty}(\Gamma_0(N);\mathbb{Q}(C))$ such that
$$
\sum_{i=1}^d A_{ij}f_i + C_{ij}g_i = -\sum_{i=1}^d\Lambda_{ij}'P_{-m_i} + D^{k+1}h_j'\text{.}
$$
By applying the functionals $\frac{k!}{m_r^{k+1}}a_{m_r}$ to the above equation, we get $Q^tA  = M$, for some $M \in M_{d\times d}(\mathbb{Q}(P))$ (recall that $\mathbb{Q}(C)\subset \mathbb{Q}(P)$). Here, $Q$ is the same matrix of Eq. (\ref{eq:poincare-base-change}); since it is invertible and has rational coefficients, we conclude that $\mathbb{Q}(A)\subset \mathbb{Q}(P)$.
\end{proof}

\section{Hecke theory; single-valued periods of modular forms}\label{sec:hecke}

In this last section, we explain how to use Hecke operators to obtain simple formulas, similar to the rank 2 case, relating the single-valued periods of $H^1_{\cusp}(\mathcal{Y}_0(N),V_k)$ to the Fourier coefficients of Poincaré series of weight $k+2$.

Recall that, for every prime $p\nmid N$, there is a correspondence
$$
\begin{tikzcd}
  & \mathcal{Y}_0(pN)\arrow{rd}{\pi_2}\arrow{ld}[swap]{\pi_1} & \\
 \mathcal{Y}_0(N) & & \mathcal{Y}_0(N)
\end{tikzcd}
$$
defined by $\pi_1 (E,C) = (E,C_N)$ and $\pi_2(E,C) = (E/C_p,C/C_p)$, where $C_N$ and $C_p$ denote the unique cyclic subgroups of $C$ of order $N$ and $p$. This induces an endomorphism $T_p$ (Hecke operator) of $H^1(\mathcal{Y}_0(N),V_k)$ in $\mathcal{H}(\mathbb{Q})$ by the usual pullback-pushforward procedure (see \cite{deligne69} (3.13)-(3.18) for details).

Given $f \in M_{k+2}^!(\Gamma_0(N))$, we have $T_{p,\dR}[f] = [T_pf]$, where
$$
T_pf \defeq \sum_{n\ge 1}(a_{pn}(f) + p^{k+1}a_{n/p}(f))q^n\text{.}
$$
By construction, the induced endomorphism $T_p$ on $H^1_{\cusp}(\mathcal{Y}_0(N),V_k)$ can be shown to be compatible with the polarisation $\langle  \ , \ \rangle$, that is, $\langle T_p \ , \ \rangle = \langle \ , T_p \ \rangle$.

Let $f \in S_{k+2}(\Gamma_0(N))$ be a normalised Hecke newform (in particular, $a_1(f)=1$ and $T_pf = a_p(f)f$ for every prime $p\nmid N$). Recall that the subfield $K_f \defeq \mathbb{Q}(a_n(f) \text{ ; } n\ge 1)\subset \mathbb{C}$ is totally real (\cite{CS17} Proposition 10.6.2).

\begin{prop}
There exists a unique rank 2 polarised subobject $H_f$ of $H_{\cusp}^1(\mathcal{Y}_0(N),V_k)\tensor_{\mathbb{Q}}K_f$ in $\mathcal{H}(K_f)$ given by $\bigcap_{p\nmid N}\ker (T_p-a_p(f)\id)$.
\end{prop}

\begin{proof}
This follows from the `multiplicity 1 theorem' (\cite{CS17} Theorem 13.3.9) and can be checked directly on the de Rham and Betti realisations. See also \cite{scholl90}.
\end{proof}

Concretely, the de Rham realisation $H_{f,\dR}$ is a $K_f$-vector space of dimension 2 admitting a basis of the form $([f],[g])$, with $g \in S_{k+2}^{!,\infty}(\Gamma_0(N);K_f)$ satisfying
\begin{enumerate}
  \item $\langle [f],[g]\rangle_{\dR}=1$, and
  \item for every prime $p\nmid N$, we have $T_pg =a_p(f)g + D^{k+1}h_{p}$, for some $h_p \in M_{-k}^{!,\infty}(\Gamma_0(N);K_f)$.
\end{enumerate}

We now wish to express the single-valued periods of $H_f$ in terms of coefficients of Poincaré series. To simplify, we shall assume from now on that
$$
S_{k+2}(\Gamma_0(N)) = S_{k+2}^{\text{new}}(\Gamma_0(N))\text{.}
$$
This happens, for instance, when $N=1$, or when $N$ is prime and $S_{k+2}(\Gamma_0(1))=0$. In this case, there exists a unique basis $(f_1,\ldots,f_d)$ of $S_{k+2}(\Gamma_0(N))$, up to ordering, where each $f_i$ is a normalised Hecke newform. If $K\subset \mathbb{R}$ is the compositum of $K_{f_1},\ldots,K_{f_d}$, then this basis induces a splitting
$$
H^1_{\cusp}(\mathcal{Y}_0(N),V_k)\tensor_{\mathbb{Q}}K \cong \bigoplus_{i=1}^dH_{f_i}\tensor_{K_{f_i}}K
$$
in $\mathcal{H}(K)$. Note that such a decomposition is necessarily orthogonal for $\langle \ , \ \rangle_{\dR}$, since, for $\omega_i \in H_{f_i,\dR}$ and $\omega_j \in H_{f_j,\dR}$, we have
$$
a_p(f_i)\langle \omega_i,\omega_j\rangle_{\dR} = \langle T_{p,\dR}\omega_i,\omega_j\rangle_{\dR} = \langle \omega_i,T_{p,\dR}\omega_j\rangle_{\dR} = a_p(f_j)\langle \omega_i,\omega_j\rangle_{\dR}
$$
for every prime $p\nmid N$.

For each $i$, let $g_i \in S_{k+2}^{!,\infty}(\Gamma_0(N);K_{f_i})$ be a weakly holomorphic modular form satisfying properties (1) and (2) above. We denote the single-valued period matrix of $H_{f_i}$ in the basis $([f_i],[g_i])$ by
$$
S_i = \left(\begin{array}{cc}
              a_i & b_i\\
              c_i & d_i
      \end{array}\right)
    $$
Then $c_i\neq 0$ and we set $\rho_i \defeq a_i/c_i$.

\begin{theorem}\label{thm:Pm-hecke-basis}
For every integer $m\ge 1$, there exists $h_m \in M_{-k}^{!,\infty}(\Gamma_0(N);K)$ such that
$$
P_m = -\frac{k!}{m^{k+1}}\sum_{i=1}^d a_{m}(f_i)c_i^{-1}f_i
$$
and
$$
P_{-m} = \frac{k!}{m^{k+1}}\left(\sum_{i=1}^d a_{m}(f_i)\rho_i f_i + a_{m}(f_i)g_i \right) + D^{k+1}h_m\text{.}
$$
\end{theorem}

\begin{proof}
  Let $\lambda_i \in \mathbb{R}$ be such that $P_m = \sum_{i=1}^d\lambda_i f_i$. Since the $f_i$ are orthogonal for $( \ , \ )_{\dR}$, Corollary \ref{coro:derham-poincare} yields
  $$
\frac{k!}{m^{k+1}}a_m(f_i) = ([f_i],[P_m])_{\dR} = \sum_{j=1}^d\lambda_j([f_i],[f_j])_{\dR} = \lambda_i([f_i],[f_i])_{\dR} = -c_i \lambda_i\text{.}
$$
This proves the first formula. By Proposition \ref{prop:poincare-betti-conjugate}, $P_{-m}$ and $-\sum_{i=1}^{d}\lambda_i(a_if_i+c_ig_i)$ are Betti conjugate, so that by Corollary \ref{coro:derahm-whmf} there exists $h_m \in M_{-k}^{!,\infty}(\Gamma_0(N))$ such that
$$
P_{-m} = \sum_{i=1}^d \lambda_i(a_if_i+c_ig_i) + D^{k+1}h_m\text{.}
$$
By taking principal parts at infinity and applying Lemma \ref{lemma:principal-part}, we conclude that the Fourier coefficients at infinity of $h_m$ lie in $K$.
\end{proof}

Taking a basis of Poincaré series, we can invert the above formulas.

\begin{theorem}
  Let $m_1,\ldots,m_d\ge 1$ be integers such that $(P_{m_1},\ldots,P_{m_d})$ is a basis of $S_{k+2}(\Gamma_0(N))$, and let $(r_{ij}) = (a_{m_j}(f_i))^{-1} \in \GL_{d}(K)$. Then, for every $1\le j \le d$, there exists $h_j' \in M_{-k}^{!,\infty}(\Gamma_0(N);K)$ such that
  $$
-c_j^{-1}f_j = \frac{1}{k!}\sum_{i=1}^dm_i^{k+1}r_{ij} P_{m_i}
$$
and
$$
\rho_j f_j + g_j = \frac{1}{k!}\left(\sum_{i=1}^dm_i^{k+1}r_{ij}P_{-m_i}\right) + D^{k+1}h'_{j}\text{.}
$$
\end{theorem}

\begin{proof}
  Let $(\lambda_{ij}) \in \GL_{d}(\mathbb{R})$ be such that $P_{m_j} = \sum_{i=1}^d\lambda_{ij} f_i$. It follows from Theorem \ref{thm:Pm-hecke-basis} that $\lambda_{ij} = -k!c_i^{-1}a_{m_j}(f_i)m_j^{-k-1}$. In matrix notation,
  $$
 \left( \begin{array}{ccc}
           &  & \\
               & \lambda_{ij}  & \\
           & &
  \end{array} \right) = -k!\left( \begin{array}{ccc}
           c_1&  & \\
               & \ddots  & \\
           & & c_d
  \end{array} \right)^{-1}\left( \begin{array}{ccc}
           &  & \\
               & a_{m_j}(f_i)  & \\
           & &
  \end{array} \right)\left( \begin{array}{ccc}
          m_1 &  & \\
               & \ddots  & \\
           & & m_d
  \end{array} \right)^{-k-1}\text{.}
  $$
  Thus, $f_j = -\frac{1}{k!}\sum_{i=1}^dm_{i}^{k+1}r_{ij}c_jP_{m_i}$ and our first formula follows. By Proposition \ref{prop:poincare-betti-conjugate}, $\rho_jf_j + g_j$ and $\sum_{i=1}^{d}m_{i}^{k+1}r_{ij}P_{-m_i}$ are Betti conjugate, so that by Corollary \ref{coro:derahm-whmf} there exists $h_j' \in M_{-k}^{!,\infty}(\Gamma_0(N))$ such that
$$
\rho_j f_j + g_j = \frac{1}{k!}\left(\sum_{i=1}^dm_i^{k+1}r_{ij}P_{-m_i}\right) + D^{k+1}h'_{j}\text{.}
$$
By taking principal parts at infinity and applying Lemma \ref{lemma:principal-part}, we conclude that the Fourier coefficients at infinity of $h'_j$ lie in $K$.
\end{proof}

We can always choose $g_j$ such that $a_1(g_j)=0$. In particular, by Proposition \ref{prop:poincare-fourier}, we obtain
\begin{equation}\label{eq:formula-rho}
\rho_j =\frac{2\pi (-1)^{\frac{k+2}{2}}}{k!}\sum_{c\ge 1\text{, } N \mid c}\left(\sum_{i=1}^d m_i^{\frac{k+1}{2}}r_{ij}\frac{K(-m_i,1;c)}{c}I_{k+1}\left(\frac{4\pi\sqrt{m_i}}{c}\right)\right) + a_1(h_j')\text{.}
\end{equation}
Let us also remark that, once the classical periods $\omega^+_j, \omega^-_j$ of $f_j$ are known, the entire period matrix
$$
P_j = \left(\begin{array}{cc}
              \omega^+_j & \eta^+_j\\
              i\omega^-_j & i\eta^-_j
            \end{array}\right)
$$
of $H_{f_j}$ is determined by $\omega^+_j,\omega^-_j,\rho_j$, since one always has the relation $\det P_j \in \mathbb{Q}^{\times}(2\pi i)^{k+1}$. As a result, we can obtain `universal' expressions for the quasi-periods of Hecke eigenforms (in the terminology of \cite{BH18}) in terms of explicit series involving Kloosterman sums and special values of Bessel functions.

\end{document}